\documentclass[11pt, reqno,sumlimits, letterpaper]{amsart}
\usepackage{amssymb,amscd, epsfig, mathrsfs, xypic, amsmath,amsthm, dsfont, txfonts} 
\xyoption{all}
\newtheorem{thm}{Theorem}[section] 
\newtheorem{cor}[thm]{Corollary}
\newtheorem{lem}[thm]{Lemma}
\newtheorem{prop}[thm]{Proposition} 
\newtheorem{df-pr}[thm]{Definition-Proposition}

\theoremstyle{definition}
\newtheorem{defn}[thm]{Definition}
   
\newtheorem{rem}[thm]{Remark}

\newtheorem{exm}[thm]{Example}

\newcommand{\RR}{{\mathbb R}}

\newcommand{\QQ}{{\mathbb Q}} 
\newcommand{\CC}{{\mathbb C}}

\newcommand{\ZZ}{{\mathbb Z}}

\newcommand{\sfb }{{\mathsf b }}

\newcommand{\sfd }{{\mathsf d}}

\newcommand{\sff }{{\mathsf f}}
\newcommand{\sfg }{{\mathsf g}}
\newcommand{\sfh }{{\mathsf h}}

\newcommand{\sfj }{{\mathsf j}}

\newcommand{\sfp }{{\mathsf p}}

\newcommand{\sfr }{{\mathsf r}}

\newcommand{\frakg}{{\mathfrak g }}

\newcommand{\frakl}{{\mathfrak   l}}

\newcommand{\frakr}{{\mathfrak  r}}

\newcommand{\frakt}{{\mathfrak  t}}

\newcommand{\calF}{{\mathcal F}}

\newcommand{\calH}{{\mathcal H}}
\newcommand{\calI}{{\mathcal I}}
\newcommand{\calJ}{{\mathcal J}}
\newcommand{\calK}{{\mathcal K}}

\newcommand{\calS}{{\mathcal S}}
\newcommand{\calT}{{\mathcal T}}

\newcommand{\calX}{{\mathcal X}}

\newcommand{\calZ}{{\mathcal Z}}

\newcommand{\scT}{{\mathscr T}}

\newcommand{\sfV }{{\mathsf V}}

\newcommand{\sfT }{{\mathsf T}}

\newcommand{\sfR }{{\mathsf R}}
\newcommand{\sfQ }{{\mathsf Q}}

\newcommand{\sfL }{{\mathsf L}}

\newcommand{\sfJ }{{\mathsf J}}

\newcommand{\sfG }{{\mathsf G}}

\newcommand{\sfD }{{\mathsf D}}
\newcommand{\sfC }{{\mathsf C}}
\newcommand{\sfB }{{\mathsf B}}
\newcommand{\sfA }{{\mathsf A}}

\newcommand{\CP}{{\mathbb C}{\mathbb P}}

\newcommand{\subt}{{\backslash}}

\newcommand{\lan}{{\langle}}
\newcommand{\ran}{{\rangle}}

\newcommand{\inc}{\hookrightarrow}

\newcommand{\sstar}{\operatorname{star}}

\newcommand{\Delt}{\operatorname{Del}}

\newcommand{\id}{{\operatorname{id}}}

\newcommand{\im}{{\operatorname{Im}}}

\newcommand{\U}{{\mbox{U}}}

\newcommand{\Mat}{{\operatorname{Mat}}}

\newcommand{\Tor}{{\operatorname{Tor}}}

\newsavebox{\savepar}

\numberwithin{equation}{section}

\newcounter{labelflag} 
\newcommand{\labelon}{\setcounter{labelflag}{1}}
\newcommand{\Label}[1]{\ifnum\thelabelflag=1\ifmmode
\makebox[0in][l]{\qquad\fbox{\rm#1}} \else
\marginpar{\vspace{0.7\baselineskip} \hspace{-1.1\textwidth}
\fbox{\rm#1}} \fi \fi \label{#1} } \labelon
\begin{document}
\title{Connected sums of simplicial complexes and equivariant cohomology}

\author{Tomoo Matsumura} \address{Algebraic Structure and its Applications Research Center, Department of Mathematical Science,  KAIST, Daejeon, 305-701, 
Republic of Korea}\email{tomoomatsumura@kaist.ac.kr}

\author{W. Frank Moore} \address{Department of Mathematics, Wake Forest
  University, Winston-Salem, NC 27106, USA} \email{moorewf@wfu.edu}
  
%\address{Max-Planck Institute of Mathematics}
%\email{tmatsu@mpim-bonn.mpg.de {\bf \today}}%
\begin{abstract}   
In this paper, we discuss the \emph{connected sum} $K_1~\#^Z K_2$ of simplicial complexes $K_1$ and $K_2$, as well as define the notion of a \emph{strong} connected sum.  Geometrically, the connected sum is motivated by Lerman's symplectic cut applied to a toric orbifold, and algebraically, it is motivated by the connected sum of rings introduced by Ananthnarayan-Avramov-Moore \cite{AAM}.

We show that the Stanley-Reisner ring of a connected sum $K_1~\#^Z K_2$ is the connected sum of the Stanley-Reisner rings of $K_1$ and $K_2$ along the Stanley-Reisner ring of $K_1 \cap K_2$.  The strong connected sum $K_1~\#^Z K_2$ is defined in such a way that when $K_1,K_2$ are Gorenstein, and $Z$ is a suitable subset of $K_1 \cap K_2$, then the Stanley-Reisner ring of $K_1 ~\#^Z K_2$ is Gorenstein, by work appearing in \cite{AAM}.  These algebraic computations can be interpreted in terms of the equivariant cohomology of moment angle complexes and we describe the symplectic cut of a toric orbifold in terms of moment angle complexes.
\end{abstract}

\maketitle
%\tableofcontents

%%%%%%%%%%%%%%%%%%%%%%%%%%%%%%%%%%%%%%%%%%
\section{{\bf Introduction}}

The \emph{moment angle complex} $\calZ_K$ associated to a simplicial complex $K$ was introduced by Buchstaber and Panov in \cite{BP99} as a disc-circle
decomposition of the Davis-Januszkiewicz universal space.  It has been actively studied in \emph{toric topology} and its connections to symplectic and algebraic geometry, and combinatorics. The original aim of introducing such a space is to generalize symplectic or algebraic toric manifolds to topological toric manifolds that are now called \emph{quasi-toric manifolds} introduced in \cite{DavisJanuszkiewicz91}.

The goals of this paper are to introduce a notion of the \emph{connected sum of simplicial complexes} by understanding the combinatorial aspect of Lerman's
symplectic cut \cite{Lerman} of a symplectic toric orbifold, and to understand the algebra structure of the (equivariant) cohomology of the corresponding moment angle complex in the framework of the \emph{connected sum of rings} introduced by Ananthnarayan-Avramov-Moore \cite{AAM}. The connected sum of simplicial complexes introduced in this paper is a more general operation than just the connected sum along a facet. 

In the first part of this paper (Section \ref{cuttingprocess}), we study a symplectic cut of a toric orbifold in terms of moment angle complexes and describe the (equivariant) cohomology ring of the toric orbifold in terms of the ones of the cut pieces, using the notion of the connected sum of rings:
\begin{thm}[Theorem \ref{main}]
Let $\calX_+$ and $\calX_-$ be the toric orbifold defined by a symplectic cut of a toric orbifold $\calX$. Let $\sfg_{\pm}: \calX_o \inc \calX_{\pm}$  be the toric sub-orbifold corresponding to the section of the cut. Let $\#$ denote  the connected sum of rings (See Definition \ref{defconn}) which is defined using the pushforward and pullback maps $\sfg_{\pm*}$ and  $\sfg_{\pm}^*$. We have
\[
H^*_{\sfR}(\calX;\ZZ) \cong H^*_{\sfR}(\calX_+, \ZZ) ~\#_{H^*_{\sfR}(\calX_o; \ZZ)}^{H^*_{\sfR}(\calX_o; \ZZ)} H^*_{\sfR}(\calX_-; \ZZ)
\]
Furhter more this descends to the non-equivariant cohomology over $\QQ$:
\[
H^*(\calX;\QQ) \cong H^*(\calX_+;\QQ) ~\#_{H^*(\calX_o; \QQ)}^{H^*(\calX_o; \QQ)} H^*(\calX_-; \QQ).
\]
This holds over $\ZZ$-coefficients if all of the cohomology rings are concentrated in even degrees.
\end{thm}
Our method is to identify the toric orbifolds as quotient stacks of moment angle complexes by a torus action and we regard the (equivariant) cohomology of toric orbifolds as the (equivariant) cohomology of moment angle complexes with appropriate torus actions. We also give a description of the cohomology ring of $\calX_-$ in terms of $\calX_+$ and $\calX$ in a similar fashion (Theorem \ref{main2}), which can be interpreted as a special case of the work previously done by Hausmann-Knutson \cite{HK}. This description is also useful, since the cutting process sometimes creates more complicated yet interesting examples.

In the second part, we introduce the \emph{connected sum of simplicial complexes} (Section \ref{secondpart}) for general simplicial complexes, abstracting the combinatorial aspect of cutting polytope by a generic hyperplane. Namely, let $K_1$ and $K_2$ be simplicial complexes on $[m]$ and let $Z \subset K_1 \cap K_2$ be a subset. We define the connected sum $K_1~\#^Z K_2$ of $K_1$ and $K_2$ by
\[
K_1~\#^Z K_2:= \Delt_Z(K_1\cup K_2) \ \ \ \ \ \ \ \ \ \ \ \ \ (\mbox{Definitions \ref{defnsimp} and \ref{defn:conn sum}}).
\]
Furthermore, we introduce the \emph{strong} connected sum of $K_1$ and $K_2$ by assuming 
\[
(\star) \ \ \ \ \ \ \ \ \ \ \ \ \ Z=K_1\backslash (\overline{K_1\backslash W}) = K_2\backslash (\overline{K_2\backslash W})
\]
where $W:=K_1\cap K_2$. We show that if $\Delta_+$  and $\Delta_-$ are simple polytopes obtained by cutting a simple polytope $\Delta$ with a generic hyperplane $H_o$, then the simplicial complex $K$ associated to $\Delta$ is a strong connected sum of the simplicial complexes $K_{\pm}$ associated to $\Delta_{\pm}$. Interestingly, $K_-$ is also a strong connected sum of $K_+$ and $K$. 
 
In Section 4,  we show that (Theorem \ref{thm2}) the Stanley-Reisner ring $\ZZ[K_1~\#^Z K_2]$ of a connected sum $K_1~\#^Z K_2$ is the connected sum of Stanley-Reisner rings $\ZZ[K_1]$ and $\ZZ[K_2]$ of $K_1$ and $K_2$ respectively, in the sense of \cite{AAM}. More explicitly, let $\sfg_i: \ZZ[K_i] \to \ZZ[W]$ and $\sff_i: \ZZ[K_1\cup K_2] \to \ZZ[K_i]$ be the natural quotient maps of Stanley-Reisner rings associated to the corresponding inclusions of simplicial complexes. Let $\calI_Z$ be the ideal in $\ZZ[W]$ generated by the monomials corresponding to elements of $Z$. Then 
\[
\ZZ[K_1\sharp^Z K_2] \cong \frac{\ker (\sfg_1 -\sfg_2: \ZZ[K_1] \times \ZZ[K_2] \to \ZZ[W])}{(\sff_1,\sff_2)(\calI_Z)}.
\]
The extra assumption ($\star$) required for the strong connected sum is motivated by the algebraic facts (see Corollary \ref{just}) that if $K_1$ and $K_2$ are Gorenstein and $W$ is Cohen-Macaulay, then the assumption ($\star$) implies that the ideal $\calI_Z$ is a canonical module of $\ZZ[W]$. As a consequence, by the work of \cite{AAM}, we can show purely algebraically that if $K_1~\#^Z K_2$ is a strong connected sum, $K_1$ and $K_2$ are Gorenstein, $W$ is Cohen-Macaulay, then $K_1~\#^Z K_2$ is Gorenstein. 

In the last section, we discuss how these algebraic structures behave if we take the torsion module of the Stanley-Reisner ring. Let $[m]=\{1,\cdots, m\}$  be the common vertex set of $K_1, K_2$ and $K$ so that the corresponding Stanley-Reisner rings are the quotients of $\ZZ[x_1,\cdots,x_m]$ by monomials given by non-faces. Let $B=(B_{ij}) \in \Mat_{n,m}(\ZZ)$ be an integral matrix of rank $n$, then we have a polynomial ring $\ZZ[\underline{u}]:=\ZZ[u_1,\cdots, u_n]$ sitting inside of $=\ZZ[x_1,\cdots,x_m]$ where $u_i=\sum_{j=1}^m B_{ij}$. In Section 4.3, we observe that if $\Tor_1^{\ZZ[\underline{u}]}(\ZZ[L], \ZZ)=0$ for $L=K,K_1,K_2,W$, then $\Tor^{\ZZ[\underline{u}]}_*(\ZZ[K_1\sharp^Z K_2],\ZZ)$ is again a connected sum of the Torsion algebras $\Tor^{\ZZ[\underline{u}]}_*(\ZZ[K_1],\ZZ)$ and $\Tor^{\ZZ[\underline{u}]}_*(\ZZ[K_2],\ZZ)$. Those torsion algebras correspond to the (equivariant) cohomology of moment angle complexes (c.f. \cite{BP}, \cite{LMM}).  The connected sum of simplicial complexes can be used to construct interesting spaces (c.f. \cite{ES}) and  the techniques developed in this paper can be used to compute the (equivariant) cohomological invariants of these spaces.

%We can interpret these results in terms of the (equivariant) cohomology of moment angle complexes. Namely, the Stanley-Reisner ring of a simplicial complex is the equivariant cohomology of the corresponding moment angle complex with respect to the canonical torus action. Taking the torsion module corresponds to considering the equivariant cohomology of the moment angle complex with respect to the action of a subgroup of the canonical torus. The connected sum is a very useful operation on simplicial complexes to create an interesting example. For instance, an example of non-starshaped simplicial complex, constructed by taking a connected sum of two star-shaped simplicial complexes. It has

\subsection*{Acknowledgements}
The authors want to thank M. Franz, T. Holm, Y. Karshon, A. Knutson, T. Ohmoto, K. Ono, D. Suh for important advice and useful conversations. The first author is particularly indebted to K. Ono for providing him an excellent environment at Hokkaido University where he had spent significant time for this paper in July and August 2011. The first author would like to show his 
gratitude to the Algebraic Structure and its Application Research Center (ASARC) at KAIST for its constant support starting 2011 September. The first author is also supported by the National Research Foundation of Korea (NRF) grant funded by the Korea government (MEST) (No. 2012-0000795, 2011-0001181).

%%%%%%%%%%%%%%%%%%%%%%%%%%%%%%%%%%%%%%%%%%

%%%%%%%%%%%%%%%%%%%%%%%%%%%%%%%%%%%%%%%%%%
%%%%%%%%%%%%%%%%%%%%%%%%%%%%%%%%%%%%%%%%%%
%%%%%%%%%%%%%%%%%%%%%%%%%%%%%%%%%%%%%%%%%%

\section{{\bf Symplectic cut of toric orbifolds}}\label{cuttingprocess}

In this section, we will first review the construction of moment angle complexes and their cohomology rings. Then we describe the symplectic cut of a toric
orbifold in terms of moment angle complexes and show the main theorem (Theorem \ref{main}) of the first part of this paper.
%%%%%%%%%%%%%%%%%%%%%%%%%%%%%%%%%%%%%%%%%%
\subsection{Moment Angle Complex} 
%%%%%%%%%%%%%%%%%%%%%%%%%%%%%%%%%%%%%%%%%%
In this section, we review the basic construction of the moment angle complexes for polytopes and general simplicial complexes. For the details, we refer to
\cite{BP} or \cite{Panov10}.
%%%%%%%%%%%%%%%%%%%%%%%%%%%%%%%%%%%%%%%%%%
\begin{defn}[c.f. p.25 \cite{BP}]\label{defnsimp}
A {\bf \emph{simplicial comlex}} on the vertex set $\calS$ is a collection $K$ of subsets (called \emph{faces}) of $\calS$ such that if $\sigma \in K$, then all subsets
including the empty $\varnothing$ of $\sigma$ are in $K$. A simplicial complex $K$ is called \emph{pure} if all its maximal faces have the same dimension where
the dimension of a face $\sigma \in K$ is $|\sigma|-1$. A maximal face is also called a \emph{facet}. The set of all facets is denoted by $\calF(K)$. A vertex $x$ is called a \emph{ghost vertex} if $\{x\} \not\in K$. Let $Z$ be a subset of a simplicial complex $K$ such that $\varnothing \not\in Z$. The \emph{closure} of $Z$ in $K$ is the smallest subcomplex containing $Z$. The \emph{open neighborhood} of $Z$ in $K$ is the set of all $\sigma \in K$ such that $\sigma$ contains some $\tau \in Z$. Note that $O_K(Z) = Z$ if and only if $K\backslash Z$ is a subcomplex of $K$. The \emph{star} of $Z$ in $K$ and the \emph{deletion} of $Z$ from $K$ are the
subcomplexes defined by $\sstar_K(Z):=\overline{O_K(Z)}$ and $\Delt_Z(K):=K \subt O_K(Z)$ respectively. If $K_1$ and $K_2$ are simplicial complexes on the same
vertex set $\calS$, then we can naturally take the intersection $K_1\cap K_2$ and the union $K_1 \cup K_2$ that are also simplicial complexes on $\calS$.
\end{defn}
%%%%%%%%%%%%%%%%%%%%%%%%%%%%%%%%%%%%%%%%%%
\begin{defn}\label{notation} Throughout this paper, we use the following notation for convenience. Let $X$ be a set and $Y, Z$ subsets of $X$. Let $\sigma \subset [m]$ be a subset. Then $Y^{\sigma}\times Z^{[m]\backslash \sigma} \subset X^m$ denotes the direct product of $Y$ and $Z$'s where $i$-th component is $Y$ if $i\in\sigma$ and $Z$ if $i\in[m]\backslash\sigma$.
\end{defn}
%%%%%%%%%%%%%%%%%%%%%%%%%%%%%%%%%%%%%%%%%%
\begin{defn}[Moment Angle Complexes]\label{macDef}
Let $K$ be a simplicial complex on the vertex set $[m]:=\{1,\cdots, m\}$ (with possible ghost vertices).  Define the \emph{moment angle complex} $\calZ_{K,[m]}\subset \CC^m$ by
\[
\calZ_{K,[m]} := \bigcup_{\sigma\in K} \sfD^{\sigma} \times \partial \sfD^{[m]\backslash\sigma} = \bigcup_{\sigma \in
  \calF(K)} \sfD^{\sigma} \times \partial \sfD^{[m]\backslash\sigma}
\]
where $\sfD=\{z\in\CC\ | \ |z| \leq 1\}$ and $\partial \sfD = \{z\in \CC \ | \ |z|=1\}$. The standard action of $\sfT:=\U(1)^m$ on $\CC^m$ can be restricted to the one on $\calZ_{K,[m]}$. \end{defn}
%%%%%%%%%%%%%%%%%%%%%%%%%%%%%%%%%%%%%%%%%%
\begin{defn}[Moment Angle Manifolds]\label{MAM}
Let $\Delta$ be a rational $n$-dimensional simple polytope in $\RR^n$ given by the inequalities:
\begin{equation}\label{inequality}
\Delta = \{\vec{x} \in \RR^n \ |\ \lan\vec{x},\lambda_i\ran + \eta_i \geq 0, i=1,\cdots,m\}, \ \ \lambda_i \in \ZZ^n, \eta_i \in \ZZ
\end{equation}
We allow this description to be ``reducible", i.e. some of the inequalities may be redundant. Or equivalently, let $H_i:=\Delta \cap \{\lan\vec{x},\lambda_i\ran +
\eta_i = 0\}$ and $H_i$ is a facet or empty. We call such an empty $H_i$ a \emph{ghost facet}.  The associated simplicial complex $K_{\Delta,[m]}$ is a
simplicial complex on $[m]$ and $\sigma \in K_{\Delta,[m]}$ if and only if $\cap_{i\in\sigma} H_i \not=\varnothing$.  Here a ghost facet corresponds to a ghost
vertex. Let $B:=[\lambda_1,\cdots,\lambda_m]$ and $\eta=(\eta_1,\cdots,\eta_m)$ and define an affine embedding $\iota_{B,\eta}: \RR^n \to \RR^m$ by
\begin{equation}\label{iota}
\iota_{B,\eta}:=B^*(\vec{x}) + \eta. 
\end{equation}
Define the \emph{moment angle manifold} $\calZ_{\Delta, B,\eta}$ for $\Delta$ given in \eqref{inequality} by the following fiber diagram:
\[
\xymatrix{
\calZ_{\Delta,B,\eta} \ar[r]^{\subset} \ar[d]& \CC^m \ar[d]_{\mu_{\sfT}}\\
\Delta \ar[r]_{\iota_{B,\eta}|_{\Delta}} & \RR^m
}
\]
where $\mu_{\sfT}(\vec{z})=(|z_1|^2,\cdots, |z_m|^2)$ is the standard moment map of the action of $\sfT:=\U(1)^m$ on $\CC^m$. It is indeed a smooth manifold
(Construction 6.8 and Lemma 6.2 \cite{BP}) and the standard $\sfT$-action on $\CC^m$ can be restricted to a $\sfT$-action on $\calZ_{\Delta, B,\eta}$. 

It is also possible to define $\calZ_{\Delta, B,\eta}$ as a quotient space. Namely, let $\sfT_{\sigma}:=U(1)^{\sigma}\times \{1\}^{[m]\backslash \sigma} \subset \sfT$ for a subset $\sigma \subset [m]$. Then there is a $\sfT$-equivariant homeomorphism $\calZ_{\Delta, B,\eta} \cong (\sfT \times \Delta)/\!\!\sim$, where $(t,p) \sim (s,q)$ if and only if $p=q$ and $ts^{-1} \in \sfT_{\sigma}$ with $p\in \cap_{i\in\sigma} H_i$.
\end{defn}
 %%%%%%%%%%%%%%%%%%%%%%%%%%%%%%%%%%%%%%%%%%
\begin{rem}[II.1 \cite{Panov10} or Section 6.1 \cite{BP}]\label{cube}
There is a $\sfT$-equivariant homeomorphism
\begin{equation}\label{Theta}
\Theta_{\Delta, B,\eta}: \calZ_{\Delta,B,\eta} \cong \calZ_{K_{\Delta},[m]}.
\end{equation}
Namely, consider a cubical subdivision of $\Delta$ defined in Construction 4.5 \cite{BP} and the corresponding decomposition of $\calZ_{\Delta, B,\eta}$:
\[
\Delta = \bigcup_{\sigma \in \calF(K_{\Delta})} C_{\sigma},\ \ \ \ \ \ \  \calZ_{\Delta, B,\eta} = \bigcup_{\sigma \in \calF(K_{\Delta})} B_{\sigma}.
\]
where $B_{\sigma}:=\mu_{\sfT}^{-1}(\iota_{B,\eta}(C_{\sigma}))$. Each $B_{\sigma}$ is $\sfT$-equivariantly homeomorphic to $\sfD^{\sigma} \times (\partial
\sfD)^{[m]\backslash \sigma}$ and these homeomorphisms are patched together to define $\Theta_{\Delta, B,\eta}$.
\end{rem}
%%%%%%%%%%%%%%%%%%%%%%%%%%%%%%%%%%%%%%%%%%
\begin{rem} We describe the parts of $\calZ_{K_{\Delta},[m]}$ corresponding to a vertex and a facet of $\Delta$ through $\Theta_{\Delta, B,\eta}$. For $\sigma \in \calF(K_{\Delta})$, let $v:=\cap_{i\in \sigma} H_i$ be a vertex of $\Delta$. Then 
\[
\Theta_{\Delta,B,\eta}(\mu_{\sfT}^{-1}(\iota_{B,\eta}(v))) = \{0\}^{\sigma} \times (\partial \sfD)^{[m]\backslash\sigma}.
\]
For a facet $H_i$ of $\Delta$, we have
\[
\Theta_{\Delta,B,\eta}(\mu_{\sfT}^{-1}(\iota_{B,\eta}(H_i))) = \bigcup_{i \in \sigma \in \calF(K_{\Delta})} \{0\}^{\{i\}} \times \sfD^{\sigma\backslash\{i\}} \times (\partial \sfD)^{[m]\backslash\sigma}.
\]
\end{rem}
%%%%%%%%%%%%%%%%%%%%%%%%%%%%%%%%%%%%%%%%%%
\begin{defn}
For a simplicial complex $K$ on $[m]$, the \emph{Stanley-Reisner ring} is defined by
\[
\ZZ[K] := \frac{\ZZ[x_1,\cdots, x_m]}{\lan x_{\sigma}, \sigma \not\in K \ran}
\]
where $x_{\sigma} := \prod_{i\in\sigma} x_i$. We identify $\ZZ[x_1,\cdots,x_m]$ with the cohomology of the classifying space of $\sfT$, $\ZZ[\sfT^*]:=H^*(B\sfT,\ZZ)$. Therefore we set $\deg x_i := 2$.
\end{defn}
%%%%%%%%%%%%%%%%%%%%%%%%%%%%%%%%%%%%%%%%%%
The basic fact about the $\sfT$-equivariant cohomology ring of $\calZ_{K,[m]}$ is
\begin{thm}[Davis-Januszkiewicz \cite{DavisJanuszkiewicz91}]\label{DJ}
There is an isomorphism of graded rings $\ZZ[K] \cong H_{\sfT}^*(\calZ_{K,[m]};\ZZ)$. This isomorphism is natural in a sense that, for a subcomplex $W\subset K$, we have the commutative diagram of short exact sequences
\[
\xymatrix{
0\ar[r]& \calI_{K\backslash W} \ar[r]\ar[d]_{\cong}& \ZZ[K] \ar[d]_{\cong} \ar[r]& \ZZ[W] \ar[d]_{\cong} \ar[r] & 0\\
0\ar[r] & H_{\sfT}^*(\calZ_{K,[m]}, \calZ_{W,[m]};\ZZ)\ar[r] & H_{\sfT}^*(\calZ_{K,[m]};\ZZ) \ar[r] & H_{\sfT}^*(\calZ_{W,[m]};\ZZ) \ar[r] & 0
}
\]
where $\calI_{K\backslash W}$ is the ideal in $\ZZ[K]$ generated by monomials $x_{\sigma}, \sigma \in K\backslash W$ and $H_{\sfT}^*(\calZ_{K,[m]}, \calZ_{W,[m]};\ZZ)$ is the relative equivariant cohomology for $\calZ_{W,[m]} \subset \calZ_{K,[m]}$ . The vertical isomorphism on the left is induced from the other two isomorphisms and the short exactness of rows.
\end{thm}
%%%%%%%%%%%%%%%%%%%%%%%%%%%%%%%%%%%%%%%%%%
%%%%%%%%%%%%%%%%%%%%%%%%%%%%%%%%%%%%%%%%%%
\subsection{Symplectic Cutting of a Toric Orbifold} In this section, to fixed the notation, we recall the construction of toric orbifolds from labeled polytopes \cite{LT} and the symplectic cut \cite{Lerman} applied to a toric orbifold.
%%%%%%%%%%%%%%%%%%%%%%%%%%%%%%%%%%%%%%%%%%

A \emph{labeled polytope} $(\Delta, \sfb)$ is an $n$-dimensional rational simple polytope $\Delta$ in $\RR^n$ where each facet $H_i, i=1,\cdots,m$ is labeled by a positive integer $\sfb_i$. Here, we assume that the $H_i$ are not ghost facets. Let $\sfT:=U(1)^m$ and $\sfR:=U(1)^n$ and $\frakt$ and $\frakr$ their Lie algebras. We identify $\frakt^*=\RR^m$ and $\frakr^*=\RR^n$. Suppose that $\Delta$ is described as 
\begin{eqnarray}\label{poly}
\Delta=\{\vec{x} \in \frakr^* \ |\ \lan \sfb_i\beta_i,\vec{x}\ran + \eta_i \geq 0, i=1, \cdots, m\}
\end{eqnarray}
where $\beta_i$ is the primitive inward normal vector to each facet $H_i$. We regard $\eta:=(\eta_1, \cdots, \eta_m)$ is an element of $\frakt^*$. Let $B$ be the integer $n\times m$ matrix defined by $B:=[\sfb_1\beta_1,\cdots, \sfb_m\beta_m]$ and regard it as the linear map $B:\frakt \to\frakr$ and also as the induced map on tori $B:\sfT \to\sfR$. The surjectivity of $B:\sfT\to\sfR$ follows from the simplicity of $\Delta$. The kernel $\sfG$ of $B: \sfT \to \sfR$ is connected if and only if $B:\ZZ^m \to \ZZ^n$ is surjective. Let $A:\sfG \to \sfT$ be the inclusion and let $A:\frakg \to \frakt$ be the induced map on the Lie algebras ($A^*:\frakt^* \to\frakg^*$).

The \emph{symplectic toric (effective) orbifold} $\calX$ for $(\Delta,\sfb)$ is given by reducing $\CC^m$ by the standard action of $\sfG$ at the regular value $A^*(\eta)$. Namely, if $\mu_{\sfT}:\CC^m \to \frakt^*$ is the standard moment map, then the moment map for the $\sfG$-action on $\CC^m$ is given by $\mu_{\sfG}:=A^*\circ \mu_{\sfT}$ and $\calX$ is defined as a quotient stack
\[
\calX := [M/\sfG], \ \ \ \ \mbox{where}\ \ \ M:=\mu_{\sfG}^{-1}(A^*(\eta)).
\]
Using the affine embedding $\iota_{B,\eta}: \frakr^* \to \frakt^*$ defined at (\ref{iota}), the moment map $\mu_{\sfR}$ for the residual $\sfR$-action on $\calX$ is given by $\mu_{\sfR}:M \stackrel{\mu_{\sfT}}{\longrightarrow} \iota_{B,\eta}(\frakr^*) \stackrel{\iota_{B,\eta}^{-1}}{\longrightarrow} \frakr^*$. Note that $\mu_{\sfT}^{-1}(\iota_{B,\eta}(\Delta))= M$ since $(A^*)^{-1}(\eta) = \iota_{B,\eta}(\frakr^*)$ and $\iota_{B,\eta}(\Delta) = \iota_{B,\eta}(\frakr^*) \cap \frakt^*_{\geq 0}$ where $\frakt^*_{\geq 0} := \mu_{\sfT}(\CC^m)$.

The symplectic cut of $\calX$ with respect to the action of a 1-dimensional subtorus $\sfL \subset \sfR$ produces two toric orbifolds $\calX_+$ and $\calX_-$ with corresponding polytopes $\Delta_+$ and $\Delta_-$ that are obtained by cutting the polytope $\Delta$ by a generic rational hyperplane
$\calH$. Let $\gamma \in \frakr$ be an integral primitive normal vector to $\calH$ and find $\xi \in \ZZ$ to write
\begin{eqnarray*}
\calH &=&\left\{ \vec{x} \in \frakr^* \ |\  \lan \gamma, \vec{x} \ran + \xi = 0\right \}\\
\Delta_+ &=& \left\{ \vec{x} \in \frakr^* \ |\  \lan \gamma, \vec{x} \ran + \xi \geq 0\right \} \cap \Delta\\
\Delta_-  &=& \left\{ \vec{x} \in \frakr^* \ |\  \lan \gamma, \vec{x} \ran + \xi  \leq 0\right\} \cap \Delta.
\end{eqnarray*}
The element $\gamma\in\frakr$ defines $1$-dimensional subtorus $\sfL:=\RR\gamma/\ZZ\gamma \subset \sfR$ and its Lie algebra $\frakl := \RR\gamma \subset \frakr$. With the natural identification $\frakl =\RR$, let $\mu: \CC \to \frakl^*$ be the standard moment map $w \mapsto |w|^2$ and let $\overline \mu: \overline \CC \to\frakl^*\ \ (w \mapsto -|w|^2)$ be the moment map for the standard $\sfL$-action on $\CC$ with the opposite symplectic structure. The \emph{symplectic cut} is to reduce $\calX \times \CC$ and $\calX \times \overline \CC$ with respect to the anti-diagonal action of $\sfL$ at the regular value $-\xi$. Namely, let $\sfd: \sfL \inc \sfR \times \sfL$ be the anti-diagonal map sending $l\mapsto (l,l^{-1})$  and consider the moment map
\begin{eqnarray*}
\varphi_+: M \times \CC \stackrel{(\mu_{\sfR},\mu)}{\longrightarrow} \frakr^* \oplus \frakl^* \stackrel{\sfd^*}{\longrightarrow} \frakl^* && (\vec{z},w) \mapsto \mu_{\sfL}(\vec{z}) - |w|^2\\
\varphi_-: M \times \overline \CC \stackrel{(\mu_{\sfR},\overline \mu)}{\longrightarrow} \frakr^* \oplus \frakl^* \stackrel{\sfd^*}{\longrightarrow} \frakl^* && (\vec{z},w) \mapsto \mu_{\sfL}(\vec{z}) + |w|^2.
\end{eqnarray*}
Then $-\xi$ is a regular value for both $\varphi_+$ and $\varphi_-$. Thus we define
\[
M_+:= \varphi_+^{-1}(-\xi), \ \ M_-:= \varphi_-^{-1}(-\xi) \ \  \mbox{ and }\ \ 
\calX_+:=[M_+/\tilde{\sfG}], \ \ \calX_-:=[M_-/\tilde{\sfG}], 
\]
where $\tilde{\sfG}$ is the preimage of $\sfd(\sfL) \subset \sfR \times \sfL$ under the map $(B,\id): \sfT \times \sfL \to \sfR \times \sfL$.

Let $\alpha:\sfR\times \sfL \to \sfR$ be defined by $\alpha(r,l):=rl$ so that $\ker \alpha = \im\ \sfd$. Define an affine embedding $\iota_{\alpha,\xi}: \frakr^* \to \frakr^*\oplus \frakl^*$ by $\iota_{\alpha,\xi}(\vec{x}):=\alpha^*(\vec{x}) + (\vec{0},\xi) = (\vec{x}, \lan \vec{x},\gamma\ran +\xi)$ so that $\iota_{\alpha,\xi}(\frakr^*) = (\sfd^*)^{-1}(-\xi)$. Then we have 
\[
M_+=(\mu_{\sfR},\mu)^{-1}(\iota_{\alpha,\xi}(\Delta_+)) \ \ \mbox{ and } \ \ M_-=(\mu_{\sfR},\overline\mu)^{-1}(\iota_{\alpha,\xi}(\Delta_-)).
\]
Thus the moment map for the $\sfR$-action on $\calX_+$ and $\calX_-$ are given by
\[
\mu_{\sfR,+}:  M_+ \stackrel{(\mu_{\sfR},\mu)}{\longrightarrow}   \iota_{\alpha,\xi}(\frakr^*) \stackrel{\iota_{\alpha,\xi}^{-1}}{\longrightarrow}
                           \frakr^* \ \ \mbox{ and } \ \ \mu_{\sfR,-}:  M_- \stackrel{(\mu_{\sfR},\overline \mu)}{\longrightarrow}
                           \iota_{\alpha,\xi}(\frakr^*) \stackrel{\iota_{\alpha,\xi}^{-1}}{\longrightarrow} \frakr^*.
\]
%%%%%%%%%%%%%%%%%%%%%%%%%%%%%%%%%%%%%%%%%%%%%%%%%%%
\subsection{$M_{\pm}$ as Quotients of Moment Angle Complexes by $\tilde{\sfG}$}
We use the notation from the previous section. Consider the integral $n\times (m+1)$ matrix $\tilde{B}:=[\sfb_1\beta_1,\cdots, \sfb_m\beta_m, \gamma]$ regarded as a map of tori $\tilde{B}:\sfT\times\sfL \to \sfR$. Then we have the commutative diagram of surjective maps
\[
\xymatrix{
\sfT \times\sfL \ar[dr]_{(B,\id)} \ar[rr]^{\tilde{B}} & & \sfR \\
 & \sfR \times \sfL \ar[ur]_\alpha &
}
\]
Since $\ker \alpha = \sfd(\sfL)$, we have $\ker \tilde{B} = \tilde{\sfG}$. Let $\tilde{A}: \tilde{\sfG} \to \sfT\times \sfL$ be the inclusion. We also denote the map on Lie algebras by $\tilde{A}:\tilde{\frakg} \to \frakt\oplus\frakl$.
%%%%%%%%%%%%%%%%%%%%%%%%%%%%%%%%%%%%%%%%%%%%%%%%%%%
\begin{lem}
$\calX_+$ and $\calX_-$ are obtained by reducing $\CC^m\times \CC$ and $\CC^m\times\overline\CC$ by the action of $\tilde{\sfG}$ at the regular value
  $\tilde{A}^*(\tilde{\eta}) \in \tilde{\frakg}$ where $\tilde{\eta}=(\eta_1,\cdots,\eta_m,\xi)$. More precisely, consider the moment maps
\[
\mu_{\tilde{\sfG},+}: \CC^m\times \CC \stackrel{(\mu_{\sfT},\mu)}{\longrightarrow} \frakt^*\oplus  \frakl^* \stackrel{\tilde{A}^*}{\to} \tilde{\frakg}^* \ \ \mbox{ and } \ \ \mu_{\tilde{\sfG},-}: \CC^m\times \overline \CC \stackrel{(\mu_{\sfT},\overline\mu)}{\longrightarrow} \frakt^*\oplus \frakl^* \stackrel{\tilde{A}^*}{\to} \tilde{\frakg}^*.
\]
Then we have
\[
M_+= \mu_{\tilde{\sfG},+}^{-1}(\tilde{A}^*(\tilde{\eta})) \ \ \mbox{ and } \ \ M_-= \mu_{\tilde{\sfG},-}^{-1}(\tilde{A}^*(\tilde{\eta})).
\]
\end{lem}
%%%%%%%%%%%%%%%%%%%%%%%%%%%%%%%%%%%%%%%%%%%%%%%%%%%
\begin{proof}
Define the affine embedding $\iota_{\tilde{B},\tilde{\eta}}: \frakr^* \to \frakt^* \oplus\frakl^*$ by $\iota_{\tilde{B},\tilde{\eta}}(\vec{x}) := \tilde{B}^*(\vec{x}) + \tilde{\eta}$ similarly as in (\ref{iota}) so that $(\tilde{A}^*)^{-1}(\tilde{A}^*(\tilde{\eta})) = \iota_{\tilde{B},\tilde{\eta}}(\frakr^*)$. We
observe that $\iota_{\tilde{B},\tilde{\eta}} = (\iota_{B,\eta}, \id) \circ \iota_{\alpha,\xi}$. Indeed,
\[
\iota_{\tilde{B},\tilde{\eta}}(\vec{x}) = \tilde{B}^*(\vec{x}) +\tilde{\eta} = (B^*(\vec{x}) +\eta, \lan \vec{x}, \gamma\ran + \xi)
                                        = (\iota_{B,\eta}(\vec{x}), \lan \vec{x}, \gamma\ran + \xi) =(\iota_{B,\eta}, \id) \circ \iota_{\alpha,\xi}(\vec{x}).
\] 
Now consider the fiber diagrams:
\[
\xymatrix{
M_+\ar[d]^{\mu_{\sfR,+}} \ar[r]_{\subset\ \ \ \ \ \ }& M_{\Delta,\sfb} \times \CC\ar[d]_{(\mu_{\sfR},\mu)}\ar[r]_{\subset}& \CC^m\times \CC\ar[d]_{(\mu_{\sfT},\mu)}\\
\frakr^* \ar[r]_{\iota_{\alpha,\xi}} \ar@/_1.5pc/[rr]_{\iota_{\tilde{B},\tilde{\eta}}}& \frakr^*\oplus \frakl^* \ar[r]_{(\iota_{B,\tilde{\eta}},\id)} &\frakt^* \oplus \frakl^*
}
 \ \ \ \ \mbox{ and }\ \ \ \   
\xymatrix{
M_-\ar[d]^{\mu_{\sfR,-} }\ar[r]_{\subset\ \ \ \ \ \ }& M_{\Delta,\sfb} \times \overline\CC\ar[d]_{(\mu_{\sfR},\overline\mu)}\ar[r]_{\subset}& \CC^m\times \overline\CC\ar[d]_{(\mu_{\sfT},\overline\mu)}\\
\frakr^* \ar[r]_{\iota_{\alpha,\xi}} \ar@/_1.5pc/[rr]_{\iota_{\tilde{B},\tilde{\eta}}}& \frakr^*\oplus \frakl^* \ar[r]_{(\iota_{B,\tilde{\eta}},\id)} &\frakt^* \oplus \frakl^*
}
\]
Since the outer circuit of each diagram is also a fiber diagram, we obtain $M_+=(\mu_{\sfT},\mu)^{-1}( \iota_{\tilde{B},\tilde{\eta}}(\frakr^*)) =
\mu_{\tilde{\sfG},+}^{-1}(\tilde{A}^*(\tilde{\eta}))$ and $M_-=(\mu_{\sfT},\overline\mu)^{-1}( \iota_{\tilde{B},\tilde{\eta}}(\frakr^*)) =
\mu_{\tilde{\sfG},-}^{-1}(\tilde{A}^*(\tilde{\eta}))$.
\end{proof}
%%%%%%%%%%%%%%%%%%%%%%%%%%%%%%%%%%%%%%%%%%
Let $K_+$ and $K_-$ be the simplicial complexed associated to $\Delta_+$ and $\Delta_-$ respectively. Here the common vertex set of $K_{\pm}$ is $\widetilde{[m]}:=[m] \cup \{o\}$. 
%%%%%%%%%%%%%%%%%%%%%%%%%%%%%%%%%%%%%%%%%%
\begin{cor}\label{+} Since $\tilde{B}$ and $\tilde{\eta}$ defines $\Delta_+$ as in (\ref{poly}), we have $M_+= \calZ_{\Delta_+, \tilde{B},\tilde{\eta}}$ as in Definition \ref{MAM}. Therefore there is a $\sfT\times\sfL$-equivariant homeomorphism $\Theta_{\Delta_+, \tilde{B},\eta}: M_+ \to \calZ_{K_+,\widetilde{[m]}}$ defined at (\ref{Theta}).
\end{cor}
%%%%%%%%%%%%%%%%%%%%%%%%%%%%%%%%%%%%%%%%%%
\begin{cor}\label{-}
There is a canonical $\sfT\times \sfL$-equivariant homeomorphism $\Psi: M_- \cong \calZ_{K_-,\widetilde{[m]}}$.
\end{cor}
\begin{proof}
The map $\sfJ:\CC^m\times \overline \CC \to \CC^m\times \CC\ \ (\vec{z},w) \mapsto (\vec{z}, \overline w)$ is a $\sfT\times \sfL$-equivariant homeomorphism with respect
to the involution $\sfj: \sfT\times\sfL \to \sfT\times \sfL, \ \ (t, l)\mapsto (t,l^{-1})$. The image $\sfJ(M_-)$ is naturally $\calZ_{\Delta_-,
  \tilde{B}',\tilde{\eta}'}$ where $\tilde{B}':=[\sfb_1\beta_1,\cdots,\sfb_m\beta_m, -\gamma]$ and $\tilde{\eta}':=(\eta_1,\cdots,\eta_m, -\xi)$. Since
$\sfJ$ also induces a $\sfT\times\sfL$-equivariant involution of $\calZ_{K_-,\widetilde{[m]}}$ with respect to $\sfj:\sfT\times\sfL \to \sfT\times\sfL$, we have an honest $\sfT\times\sfL$-equivariant homeomorphism:
\[
\Psi: M_- \stackrel{\sfJ}{\longrightarrow} \sfJ(M_-) = \calZ_{\Delta_-, \tilde{B}',\tilde{\eta}'}
\stackrel{\Theta_{\Delta_-,\tilde{B}',\tilde{\eta}'}}{\longrightarrow} \calZ_{K_-,\widetilde{[m]}} \stackrel{\sfJ}{\longrightarrow} \calZ_{K_-,\widetilde{[m]}}.
\]
\end{proof}
%%%%%%%%%%%%%%%%%%%%%%%%%%%%%%%%%%%%%%%%%%
\begin{cor}\label{214}
Topologically $\calX_+\cong[\calZ_{K_+,\widetilde{[m]}}/\tilde{\sfG}]$ and $\calX_-\cong[\calZ_{K_-,\widetilde{[m]}}/\tilde{\sfG}]$
\end{cor}
%%%%%%%%%%%%%%%%%%%%%%%%%%%%%%%%%%%%%%%%%%
%%%%%%%%%%%%%%%%%%%%%%%%%%%%%%%%%%%%%%%%%%
\subsection{Gluing along the toric suborbifold}
%%%%%%%%%%%%%%%%%%%%%%%%%%%%%%%%%%%%%%%%%%
%%%%%%%%%%%%%%%%%%%%%%%%%%%%%%%%%%%%%%%%%%

Let $H_o=\Delta_+\cap\Delta_- \subset \frakr^*$.  Consider the obvious inclusions $\sfh_+: \CC^m\times \{0\} \to \CC^m\times \CC$ and $\sfh_-:\CC^m\times
\{0\} \to \CC^m\times \overline \CC$. Let $M_o^+:=(\mu_{\sfT},\mu)^{-1}( \iota_{\tilde{B},\tilde{\eta}}(H_o)) \subset \im\ \sfh_+$ and
$M_o^-:=(\mu_{\sfT},\overline\mu)^{-1}( \iota_{\tilde{B},\tilde{\eta}}(H_o)) \subset \im \ \sfh_-$. Define the suborbifold corresponding to $H_0$ in $\calX_+$ and
$\calX_-$ by
\[
\calX_o:= [M_o/\tilde{\sfG}] \ \ \ \mbox{ where }  M_o:=\sfh_+^{-1}(M_o^+) = \sfh_-^{-1}(M_o^-).
\]
together with the embedding $\sfh_+: M_o \inc M_+$ and $\sfh_-: M_o \inc M_-$. We obtained the space $M_+\cup_{M_o} M_-$ which is given by gluing $M_+$ and $M_-$
along $M_o$ with respect to $\sfh_+$ and $\sfh_-$.
%%%%%%%%%%%%%%%%%%%%%%%%%%%%%%%%%%%%%%%%%%

On the other hand, since $K_+$ and $K_-$ have the common vertex set $\widetilde{[m]}$, we can naturally glue them to obtain a simplicial complex $K_+ \cup K_-$ where $W:=K_+\cap K_- = \sstar_{K_+}(o) = \sstar_{K_-}(o)$ where $\sstar_{K_{\pm}}(o)$ is the smallest simplicial complex containing all faces in $K_{\pm}$ that contain $o$. It follows from Definition \ref{macDef} that $\calZ_{K_+\cup K_-} = \calZ_{K_+} \cup \calZ_{K_-}$ and $\calZ_{W}=\calZ_{K_+} \cap \calZ_{K_-}$ where we suppressed the vertex set $\widetilde{[m]}$. The image of $M_o$ under $\Theta_{\Delta_+,\tilde{B},\tilde{\eta}}$ and $\Psi$ coincide with
\[
\calZ_{W}^{\circ} := \bigcup_{o \in \sigma \in \calF(K_+)} \{0\}^{\{o\}} \times \sfD^{\sigma\backslash \{o\}} \times (\partial\sfD)^{\widetilde{[m]}\backslash
  \sigma} =\{0\}^{\{o\}} \times \left( \bigcup_{o \in \sigma \in \calF(K_+)} \sfD^{\sigma\backslash \{o\}} \times \sfD^{\widetilde{[m]}\backslash \sigma}
\right).
\]
It is a subspace of 
\[
\calZ_{W} = \bigcup_{\sigma \in \calF(W)} \sfD^{\sigma}\times (\partial \sfD)^{\widetilde{[m]} \backslash \sigma} = \sfD^{\{o\}} \times \left( \bigcup_{o \in
  \sigma \in \calF(W)} \sfD^{\sigma\backslash \{o\}} \times \sfD^{\widetilde{[m]}\backslash \sigma} \right).
\]
Therefore the $\tilde{\sfT}$-equivariant homeomorphism $\Theta_+:=\Theta_{\Delta_+,\tilde{B},\tilde{\eta}}$ and $\Psi$ induces a $\tilde{\sfT}$-equivariant map
\[
\Phi: M_+\cup_{M_o} M_- \to \calZ_{K_+ \cup K_-}.
\]
%%%%%%%%%%%%%%%%%%%%%%%%%%%%%%%%%%%%%%%%%%
\begin{lem}\label{lemUnion}
For any subgroup $\sfQ \subset \tilde{\sfT}$, the pullback $\Phi^*: H_{\sfQ}^*(\calZ_{K_+\cup K_-},\ZZ) \to H_{\sfQ}^*(M_+ \cup_{M_o} M_-,\ZZ)$
is an isomorphism.
\end{lem}
%%%%%%%%%%%%%%%%%%%%%%%%%%%%%%%%%%%%%%%%%%
\begin{proof}
We observe that there is a $\tilde{\sfT}$-equivariant deformation retract from $\calZ_W$ to $\calZ_W^{\circ}$, therefore $\Phi|_{M_o}^*: H_{\sfQ}^*(\calZ_{W}) \cong H_{\sfQ}^*(M_o)
$. The claim follows from the diagram of the Mayer-Vietoris sequences and the Five Lemma:
\[
\xymatrix{ 
H_{\sfQ}^{*-1}(M_+) \!\oplus\! H_{\sfQ}^{*-1}(M_-) \ar[r]^{\ \ \ \ \ \ \  \sfh_{+}^*-\sfh_{-}^*}  &H_{\sfQ}^{*-1}(M_o) \ar[r]& H_{\sfQ}^*(M_+ \cup_{M_o} M_-) \ar[r] & H_{\sfQ}^*(M_+)\! \oplus\! H_{\sfQ}^*(M_-) \ar[r]^{\ \ \ \ \ \ \  \sfh_{+}^*-\sfh_{-}^*} & H_{\sfQ}^*(M_o)\\ 
H_{\sfQ}^{*-1}(\calZ_{K_+}) \!\oplus\! H_{\sfQ}^{*-1}(\calZ_{K_-}) \ar[u]_{\cong}^{(\Theta_+^*,\Psi^*)}\ar[r]&H_{\sfQ}^{*-1}(\calZ_{W}) \ar[r] \ar[u]_{\cong}& H_{\sfQ}^*(\calZ_{K_+ \cup K_-}) \ar[u]_{\Phi^*} \ar[r] & H_{\sfQ}^*(\calZ_{K_+})\! \oplus\! H_{\sfQ}^*(\calZ_{K_-}) \ar[u]_{\cong}^{(\Theta_+^*,\Psi^*)}\ar[r] & H_{\sfQ}^*(\calZ_{W}) \ar[u]_{\cong} }
\]
\end{proof}
%%%%%%%%%%%%%%%%%%%%%%%%%%%%%%%%%%%%%%%%%%
\begin{lem}
\end{lem}
%%%%%%%%%%%%%%%%%%%%%%%%%%%%%%%%%%%%%%%%%%
\subsection{Computing the Cohomology of $\calX$} \label{M}
The original toric orbifold $\calX$ can also be defined by adding one more trivial inequality for $\Delta$:
\[
\lan \vec{x},\gamma\ran  + \xi' \geq 0, \ \ \ \xi' \gg 0. 
\]
Let $\tilde{\eta}':=(\eta_1,\cdots, \eta_m, \xi')$ and reduce $\CC^m\times \CC$ by the action of $\tilde{\sfG}$ at the regular value $\tilde{A}^*(\tilde{\eta}')$. We have
\[
\calX = [M'/\tilde{\sfG}] \ \ \mbox{ where } \ \ M':=\mu_{\tilde{\sfG}}^{-1}(\tilde{A}^*(\tilde{\eta}')).
\]
Then $M'=\calZ_{\Delta, \tilde{B},\tilde{\eta}'}$ and so we have the $\tilde{\sfT}$-equivariant homeomorphism $\Theta:=\Theta_{\Delta, \tilde{B},\tilde{\eta}'}: M' \to
\calZ_{K_{\Delta},\widetilde{[m]}}$. Thus we can identify $\calX \cong [\calZ_{K_{\Delta},\widetilde{[m]}}/\tilde{\sfG}]$.

Now for any subgroup $\sfQ \subset \tilde{\sfT}$, there are two long exact sequences to compute the (equivariant) cohomology of $M'\cong \calZ_{K_{\Delta},\widetilde{[m]}}$. One is the Mayer-Vietoris Sequence as in the proof of Lemma \ref{lemUnion} and the other is the relative cohomology sequence
\begin{equation}\label{relative}
\xymatrix{
\cdots \ar[r] &H^*_{\sfQ}(\calZ_{K_+\cup K_-}, \calZ_{K_{\Delta}}) \ar[r]_{\sfr_1^*}& H^*_{\sfQ}(\calZ_{K_+\cup K_-}) \ar[r]_{\sfr_2^*}& H^*_{\sfQ}(\calZ_{K_{\Delta}})\ar[r]& \cdots
}
\end{equation}
%%%%%%%%%%%%%%%%%%%%%%%%%%%%%%%%%%%%%%%%%%
Note that there is an isomorphism $\scT: H^{*-2}_{\sfQ}(\calZ_W) \to H^*_{\sfQ}(\calZ_{K_+\cup K_-}, \calZ_K)$ defined through the Thom isomorphism for $\calZ_W^{\circ}\subset \calZ_W$ and obvious pullback maps:
\[
\xymatrix{ 
H^{*-2}_{\sfQ}(\calZ_W)\ar[r]_{\cong} & H^{*-2}_{\sfQ}(\calZ_W^{\circ})\ar[r]_{\cong\ \ \ \ \ \ \ }^{T\!hom\ \ \ \ \ \ \ } &H^*_{\sfQ}(\calZ_W, \calZ_W\backslash \calZ_W^{\circ})\ar[r]_{\cong} & H^*_{\sfQ}(\calZ_W, \calZ_{\Delt_oW})& H^*_{\sfQ}(\calZ_{K_+\cup K_-},\calZ_{K_{\Delta}})
  \ar[l]^{\ \ \ \ \cong}.\\
  }
\]
Furthermore, we also have the natural maps $\scT_{\pm}: H^{*-2}_{\sfQ}(\calZ_W) \to H^*_{\sfQ}(\calZ_{K_{\pm}})$ given as a composition of $\calT$ and obvious pullback maps:
\[
\xymatrix{ 
\scT_{\pm}:H^{*-2}_{\sfQ}(\calZ_W)\ar[r]^{\scT\ \ \ \ \ }_{\cong\ \ \ \ }& H^*_{\sfQ}(\calZ_{K_+\cup K_-}, \calZ_{K_{\Delta}}) \ar[r]_{\sfr_1^*}&H_{\sfQ}^*(\calZ_{K_+\cup K_-}) \ar[r] & H_{\sfQ}^*(\calZ_{K_{\pm}}).\\
  }
\]
If the Mayer-Vietoris sequence and the relative cohomology sequence split into short exact sequences, more precisely, if the odd degrees of the cohomology of $\calZ_W, \calZ_{K_{\pm}}$ and $\calZ_{K_{\Delta}}$  vanish, then $H_{\sfQ}^*(\calZ_{K_{\Delta}})$ is isomorphic to the quotient of the kernel of $H_{\sfQ}^*(\calZ_{K_+ \cup K_-}) \to H_{\sfQ}^*(\calZ_{K_+})\! \oplus\! H_{\sfQ}^*(\calZ_{K_-})$ by the image of $(\scT_+,\scT_-)$. Since $\scT_{\pm}$ can be identified with the pushforward maps $\sfh_{\pm+}$ respectively, we also have that $H_{\sfQ}^*(M')$ is isomorphic to the quotient of the kernel of $\sfh_+^*-\sfh_-^*$ by the image of $(\sfh_{+*},\sfh_{-*})$. We state this result for the case that we are interested in:
%%%%%%%%%%%%%%%%%%%%%%%%%%%%%%%%%%%%%%%%%%
\begin{thm}\label{main}
Recall the embedding $\sfh_{\pm}:M_o \to M_{\pm}$. We have
\[
H_{\tilde{\sfT}}^*(M';\ZZ) \cong \frac{\ker \left( \sfh_+^*-\sfh_-^*: H_{\tilde{\sfT}}^*(M_+;\ZZ) \oplus H_{\tilde{\sfT}}^*(M_-;\ZZ) \to
  H_{\tilde{\sfT}}^*(M_o;\ZZ) \right)}{ \im \left((\sfh_{+*}, \sfh_{-*}): H_{\tilde{\sfT}}^*(M_o;\ZZ) \to H_{\tilde{\sfT}}^*(M_+;\ZZ) \oplus
  H_{\tilde{\sfT}}^*(M_-;\ZZ) \right) }
\]
Furthermore
\[
H_{\tilde{\sfG}}^*(M';\QQ) \cong \frac{\ker \left( \sfh_+^*-\sfh_-^*: H_{\tilde{\sfG}}^*(M_+;\QQ) \oplus H_{\tilde{\sfG}}^*(M_-;\QQ) \to
  H_{\tilde{\sfG}}^*(M_o;\QQ) \right)}{ \im \left((\sfh_{+*}, \sfh_{-*}): H_{\tilde{\sfG}}^*(M_o;\QQ) \to H_{\tilde{\sfG}}^*(M_+;\QQ) \oplus
  H_{\tilde{\sfG}}^*(M_-;\QQ) \right) },
\]
which is also true over $\ZZ$-coefficients if the cohomology rings of $M_o,M_{\pm}, M'$ are concentrated in even degrees.
\end{thm}
\begin{proof}
The first claim follows, since the odd degree of $\tilde{\sfT}$-equivariant cohomology vanishes. The second claim follows from the fact that the odd degree of rational ordinary cohomology of toric orbifolds vanishes \cite{Danilov78,Jurkiewicz85}.
\end{proof}
%%%%%%%%%%%%%%%%%%%%%%%%%%%%%%%%%%%%%%%%%%
\begin{rem}
Let $\sfT$ act on $M$ and suppose the action of $\sfG \subset \sfT$ is locally free. This defines an $\sfR:=\sfT/\sfG$-action on an orbifold $[M/\sfG]$. The cohomology
$H^*([M/\sfG];\ZZ)$ is defined to be $H^*_{\sfG}(M;\ZZ)$ and the equivariant cohomology $H_{\sfR}^*([M/\sfG];\ZZ)$ is defined to be $H_{\sfT}^*(M;\ZZ)$. We
refer to Edidin \cite{Edidin10} for the details. With the notation of the connected sum of rings which is explained in Definition \ref{defconn}, Theorem
\ref{main} is exactly our main theorem described in the introduction.
\end{rem}
%%%%%%%%%%%%%%%%%%%%%%%%%%%%%%%%%%%%%%%%%%
%%%%%%%%%%%%%%%%%%%%%%%%%%%%%%%%%%%%%%%%%%
\subsection{Computing Cohomology of $\calX_-$}
%%%%%%%%%%%%%%%%%%%%%%%%%%%%%%%%%%%%%%%%%%
%%%%%%%%%%%%%%%%%%%%%%%%%%%%%%%%%%%%%%%%%%
Similarly we can consider the following two long exact sequences in terms of moment angle complexes and interpret them in terms of level sets of moment maps. Again we suppress the vertex set $\widetilde{[m]}$ from the notation of moment angle complexes. Let $\tilde{K}:=K_+\cup K_- = K \cup K_+$.  We have the Mayer-Vietoris Sequence
\begin{equation}\label{MV2}
\cdots \to H_{\sfQ}^{*-1}(\calZ_{K_+\cap K}) \to H_{\sfQ}^*(\calZ_{\tilde{K}}) \to H_{\sfQ}^*(\calZ_{K_+}) \oplus H_{\sfQ}^*(\calZ_{K}) \to
H_{\sfQ}^*(\calZ_{K_+\cap K}) \to H_{\sfQ}^{*+1}(\calZ_{\tilde{K}}) \to \cdots;
\end{equation}
and the relative cohomology sequence
\begin{equation}\label{relative2}
\cdots \to H^{*-1}_{\sfQ}(\calZ_{K_-})\to H^*_{\sfQ}(\calZ_{\tilde{K}}, \calZ_{K_-}) \to H^*_{\sfQ}(\calZ_{\tilde{K}}) \to H^*_{\sfQ}(\calZ_{K_-}) \to
H^{*+1}_{\sfQ}(\calZ_{\tilde{K}}, \calZ_K) \to \cdots.
\end{equation}
Let $\tilde{B}$, $\tilde{\eta}'$ and $M'$ be the ones defined in Section \ref{M}. Let $N_+:=(\mu_{\sfT},\mu)^{-1}(\iota_{\tilde{B},\tilde{\eta}'}(\Delta_+))$. Since $\Delta_+ \subset \Delta$, we have the obvious inclusion $\sff: N_+ \subset M'$. We can choose a cubic subdivision of $\Delta$ in such a way that $\Theta_{\Delta, \tilde{B},\tilde{\eta}'}(N_+) = \calZ_{K\cap K_+}$. Let $\sfg_+: N_+ \to M_+$ be the natural inclusion defined by $N_+ \cong \calZ_{K\cap K_+} \inc \calZ_{K_+} \cong M_+$. Thus the map $H_{\sfQ}^*(\calZ_{K_+}) \oplus H_{\sfQ}^*(\calZ_{K}) \to
H_{\sfQ}^*(\calZ_{K_+\cap K})$ in (\ref{MV2}) can be replaced by 
\[
H_{\sfQ}^*(M_+) \oplus H_{\sfQ}^*(M') \stackrel{\sfg_+^*-\sff^*}{\longrightarrow}
H_{\sfQ}^*(N_+);
\]

On the other hand, observe that the inclusions of pairs $(K,K\cap K_-) \subset (\tilde{K},K_-) \supset (K_+, W) \supset (K_+\cap K, W\cap K)$ induces isomorphism by pullback on relative cohomology:
\[
H^*_{\sfQ}(\calZ_{\tilde{K}}, \calZ_{K_-}) \cong H^*_{\sfQ}(\calZ_K, \calZ_{K\cap K_-}) \cong H^*_{\sfQ}(\calZ_{K_+}, \calZ_W) \cong H^*_{\sfQ}(\calZ_{K_+\cap K}, \calZ_{W\cap K}).
\]
Let $N_-:=(\mu_{\sfT},\mu)^{-1}(\iota_{\tilde{B},\tilde{\eta}'}(\Delta_-))$ and $N_o:=(\mu_{\sfT},\mu)^{-1}(\iota_{\tilde{B},\tilde{\eta}'}(H_0))$ so that, with the same cubic subdivision of $\Delta$ used above, we have $\Theta_{\Delta, \tilde{B},\tilde{\eta}'}(N_-) = \calZ_{K\cap K_-}$ and $\Theta_{\Delta, \tilde{B},\tilde{\eta}'}(N_o) = \calZ_{K\cap W}$. Then by the functoriality, the map $H^*_{\sfQ}(\calZ_{\tilde{K}}, \calZ_{K_-}) \to H^*_{\sfQ}(\calZ_{\tilde{K}})  \to H_{\sfQ}^*(\calZ_{K_+}) \oplus
H_{\sfQ}^*(\calZ_{K})$ can be replaced by the following map:
\[
\delta: H^*_{\sfQ}(N_+, N_o) \stackrel{diag}{\longrightarrow} H^*_{\sfQ}(N_+, N_o) \oplus H^*_{\sfQ}(N_+, N_o) \cong H^*_{\sfQ}(M_+,
\Theta_{\Delta_+,\tilde{B},\tilde{\eta}}^{-1}(\calZ_W)) \oplus H^*_{\sfQ}(M, N_-) \to H_{\sfQ}^*(M_+) \oplus H_{\sfQ}^*(M).
\]
Thus similarly to Theorem \ref{main}, we obtain the following theorem:
\begin{prop}[c.f. \cite{HK}]\label{main2} We have
\[
H_{\tilde{\sfT}}^*(M_-;\ZZ) \cong \frac{\ker \left( (\sfg_+^*,-\sff^*): H_{\tilde{\sfT}}^*(M_+;\ZZ) \oplus H_{\tilde{\sfT}}^*(M;\ZZ) \to
  H_{\tilde{\sfT}}^*(N_+;\ZZ) \right)}{ \im \left(\delta: H_{\tilde{\sfT}}^*(N_+,N_o;\ZZ) \to H_{\tilde{\sfT}}^*(M_+;\ZZ) \oplus H_{\tilde{\sfT}}^*(M;\ZZ)
  \right) }.
\]
Furthermore, if $H_{\tilde{\sfG}}^*(M;\ZZ) \to H_{\tilde{\sfG}}^*(N_+;\ZZ)$ or $H_{\tilde{\sfG}}^*(M_+;\ZZ) \to H_{\tilde{\sfG}}^*(N_+;\ZZ)$ is surjective, then 
\[
H_{\tilde{\sfG}}^*(M_-;\ZZ) \cong \frac{\ker \left( (\sfg_+^*,-\sff^*): H_{\tilde{\sfG}}^*(M_+;\ZZ) \oplus H_{\tilde{\sfG}}^*(M;\ZZ) \to
  H_{\tilde{\sfG}}^*(N_+;\ZZ) \right)}{ \im \left(\delta: H_{\tilde{\sfG}}^*(N_+,N_o;\ZZ) \to H_{\tilde{\sfG}}^*(M_+;\ZZ) \oplus H_{\tilde{\sfG}}^*(M;\ZZ)
  \right) }
\]
\end{prop}
%%%%%%%%%%%%%%%%%%%%%%%%%%%%%%%%%%%%%%%%%%
\begin{rem}
The above proposition is a special case of what is proved by Hausmann-Knutson \cite{HK} for more general symplectic cuts. They used the projection $\sfp:N_+ \to M_+$ by quotienting the boundary of
$N_+$ by a circle action, instead of the inclusion $\sfg_+: N_+ \inc M_+$ in our case. It is actually easy to see that $\sfp$ and $\sfg_+$ are homotopy equivalent. Namely,
\[
N_+\cong \calZ_{K\cap K_+,\widetilde{[m]}} = (\partial \sfD)^{\{o\}} \times \calZ_{K\cap K_+, [m]}. 
\]
is a deformation retract of
\[
N_+^{\bullet}\cong \calZ_{K\cap K_+,\widetilde{[m]}}^{\bullet} = (\sfD\backslash\frac{1}{2}\sfD)^{\{o\}} \times \calZ_{K\cap K_+, [m]}
\]
where $\sfD\backslash\frac{1}{2}\sfD = \{ z \in \CC \ |\ \frac{1}{2} \leq |z|\leq 1\}$. Define $h_t: N_+^{\bullet} \to \calZ_{K_+,\widetilde{[m]}}, 0\leq t \leq
1$ by sending $\sfD\backslash\frac{1}{2}\sfD \to \sfD$ via
\[
re^{2\pi i\theta} \mapsto \left(\frac{1}{1+t}\right)\left(r- \frac{1}{2}\right)e^{2\pi i\theta}.
\]

\end{rem}
%%%%%%%%%%%%%%%%%%%%%%%%%%%%%%%%%%%%%%%%%%
%%%%%%%%%%%%%%%%%%%%%%%%%%%%%%%%%%%%%%%%%%
\section{{\bf Connected sum of simplicial complexes}}\label{secondpart}
In this section, we define the (strong) connected sum $K_1~\#^Z K_2$ of simplicial complexes $K_1$ and $K_2$ on a vertex set $[m]$. It is motivated
by the simplicial complexes of the polytopes obtained by the symplectic cut of a toric orbifold. We show that the case of the cutting polytope defines a strong
connected sum of simplicial complexes.
%%%%%%%%%%%%%%%%%%%%%%%%%%%%%%%%%%%%%%%%%%
%%%%%%%%%%%%%%%%%%%%%%%%%%%%%%%%%%%%%%%%%%
%%%%%%%%%%%%%%%%%%%%%%%%%%%%%%%%%%%%%%%%%%

%%%%%%%%%%%%%%%%%%%%%%%%%%%%%%%%%%%%%%%%%%
%%%%%%%%%%%%%%%%%%%%%%%%%%%%%%%%%%%%%%%%%%
\subsection{(Strong) Connected Sums} 
%%%%%%%%%%%%%%%%%%%%%%%%%%%%%%%%%%%%%%%%%%
%%%%%%%%%%%%%%%%%%%%%%%%%%%%%%%%%%%%%%%%%%
\begin{defn}[Connected Sum]\label{defn:conn sum}
Recall our notation from Definition \ref{defnsimp}. Let $K_1$ and $K_2$ be simplicial complexes on $[m]$. Let $Z \subset K_1\cap K_2$ be a subset not
containing the empty set and suppose that $O_{K_1\cup K_2}(Z) \subset K_1\cap K_2$. The \emph{connected sum} $K_1 ~\#^Z K_2$ of $K_1$ and
$K_2$ along $Z$ is defined by
\[
K_1 ~\#^Z K_2 := \Delt_Z(K_1\cup K_2).
\]
Note that since $O_K(Z) \subset K_1\cap K_2$ and $K_1\cap K_2$ is a subcomplex, $\sstar_K(Z)=\overline{O_K(Z)} \subset K_1\cap K_2$.
\end{defn}
%%%%%%%%%%%%%%%%%%%%%%%%%%%%%%%%%%%%%%%%%%
\begin{exm}[Connected sum along a facet p.24 \cite{BP}]
Let $K_1$ and $K_2$ be two pure simplicial complexes. Let $\sigma_i \in \calF(K_i)$. If we identify the vertex sets of $\sigma_1$ and $\sigma_2$, we have $K_1
\cap K_2 = \overline{\sigma}$ where we denote $\sigma = \sigma_1 = \sigma_2$. Let $Z:=\{\sigma\}$ and then $O_{K_1\cup K_2}(Z)=\{\sigma\} \subset K_1\cap
K_2$. The connected sum $K_1 ~\#^{\sigma} K_2 := K_1 \cup K_2 \backslash \{\sigma\}$ is exactly the ``connected sum" defined in \cite{BP}.
\end{exm}
%%%%%%%%%%%%%%%%%%%%%%%%%%%%%%%%%%%%%%%%%%
\begin{exm}
Let $v(K_1)=\{a,b,c,d\}$ and $v(K_2)=\{a,b,c,e\}$. Let $\calF(K_1)=\{abc,bcd\}$ and $\calF(K_2)=\{abc, ace\}$. Then $\calF(W)=\{abc\}$ and let
$Z=\{abc\}=O_K(Z)$. This is a connected sum which is a connected sum in the sense of \cite{BP}. The result is not pure.
\end{exm}
%%%%%%%%%%%%%%%%%%%%%%%%%%%%%%%%%%%%%%%%%%
The \emph{strong connected sum} is a connected sum with an extra condition on the part $Z$ we delete from the union $K_1\cup K_2$. The algebraic
justification comes in the later section and here we show the following lemma.
%%%%%%%%%%%%%%%%%%%%%%%%%%%%%%%%%%%%%%%%%%
\begin{lem}\label{basiclemma} Let $W$ be a subcomplex of a simplicial complex $K$. Let
\begin{equation}\label{partZ}
Z := \{ \tau \in K \ |\ \tau\cup\sigma \not\in K, \forall \sigma \in K\backslash W\}. 
\end{equation}
Then $O_K(Z)=Z$ and $Z = W \backslash (\overline{K\backslash W})$. 
\end{lem}
%%%%%%%%%%%%%%%%%%%%%%%%%%%%%%%%%%%%%%%%%%
\begin{proof}
By definition, if $\tau \in O_K(Z)$, then there is $\tau' \in Z$ such that $\tau' \subset \tau$. Thus for all $\sigma \in K\backslash W$, $\sigma \cup \tau
\not\in K$, because if otherwise $\sigma \cup \tau' \in K$. This shows $O_K(Z) = Z$. To show $Z = W \backslash (\overline{K\backslash W})$, first observe that
$Z \subset W$. Indeed, if $\tau \in K\backslash W$, then $\tau \cup \tau =\tau \in K$ and so $\tau \not\in Z$. If $\tau \in \overline{K\backslash W}$, then
there is $\sigma \in K\backslash W$ such that $\tau \subset \sigma$ and so $\tau \cup \sigma =\sigma \in K$. Thus $Z \subset W \backslash (\overline{K\backslash
  W})$. On the other hand, let $\tau \in W \backslash \overline{K\backslash W}$. If $\tau \not \in Z$, then there is $\sigma \in K\backslash W$ such that $\tau
\cup \sigma \in W$. This means $\tau \in \sstar_K(K\backslash W)$. However, recall from Definition \ref{defnsimp} that $\sstar_K(K\backslash W) =
\overline{O_K(K\backslash W)} = \overline{K\backslash W}$. Thus $\tau \in \overline{K\backslash W}$ which is a constradiction. Thus $\tau \in Z$ and so $W
\backslash\overline{K\backslash W} \subset Z$.
\end{proof}
%%%%%%%%%%%%%%%%%%%%%%%%%%%%%%%%%%%%%%%%%%
\begin{defn}[Strong connected sum]\label{sconnected}
A connected sum $K_1~\#^Z K_2$ is called \emph{strong} if $K_1, K_2$ and $K_1\cap K_2$ are pure with the same dimension and 
\[
Z = W \backslash (\overline{K_1\backslash W}) = W \backslash (\overline{K_2\backslash W})
\]
Algebraic justification of the following definition will be explained in Section \ref{canonical}.
\end{defn}
%%%%%%%%%%%%%%%%%%%%%%%%%%%%%%%%%%%%%%%%%%
%%%%%%%%%%%%%%%%%%%%%%%%%%%%%%%%%%%%%%%%%%
\subsection{Polytope cutting and connected sum} 
%%%%%%%%%%%%%%%%%%%%%%%%%%%%%%%%%%%%%%%%%%
%%%%%%%%%%%%%%%%%%%%%%%%%%%%%%%%%%%%%%%%%%
\begin{defn}[c.f. Section 1.1 \cite{BP}]
A \emph{polytope} $\Delta$ is defined to be the convex hull of a finite set of points in $\RR^n$. Suppose that 
\[
\Delta= \{ \vec{x} \in \RR^n \ | \ \lan \vec{x}, \lambda_i\ran + \eta_i \geq 0, i=1,\cdots,m\}. 
\]
for some $\lambda_i\in (\RR^n)^*$ and $\lambda_i \in \RR$. A polytope $\Delta$ is \emph{simple} if the bounding hyperplanes $\tilde{H}_i:=\{\lan \vec{x},
\lambda_i\ran + \eta_i = 0\}$ are in general position, i.e. if the dimension of $\Delta$ is $r$, then there are exactly $r$ hyperplanes $\tilde{H}_i$ meeting at
each vertex of $\Delta$. We call $H_i:=\Delta \cap \tilde{H}_i$ a \emph{facet} for each $i=1,\cdots,m$. Note that $H_i$ is $r-1$ dimensional or empty. If $H_i$
is empty, we call it a \emph{ghost facet}.

For a simple polytope $\Delta$ with facets $H_i, i=1,\cdots,m$, the associated simplicial complex $K_{\Delta}$ is a simplicial complex on $[m]$ defined by 
\[
\sigma \subset K_{\Delta} \Leftrightarrow \sigma=\varnothing \ \mbox{ or } \  \bigcap_{i\in\sigma} H_i \not=\varnothing.
\]
\end{defn}
%%%%%%%%%%%%%%%%%%%%%%%%%%%%%%%%%%%%%%%%%%
\begin{defn}[Generic cut]\label{cut}
Let $\Delta \subset \RR^n$ be a $n$-dimensional simple polytope with non-ghost facets $H_i, i=1,\cdots,m$. Consider a hyperplane
\[
\calH := \{\vec{x} \in \RR^n\ |\ \lan\vec{x}, \lambda_0 \ran +\xi = 0\}
\]
and the corresponding closed half spaces $\tilde{H}_{+} = \{\lan\vec{x}, \lambda_0 \ran +\xi \geq 0\}$ and $\tilde{H}_{-} = \{\lan\vec{x}, \lambda_0 \ran
+\xi \leq 0\}$. A \emph{generic cut} of $\Delta$ is given by the pair $(\Delta, \calH)$ such that $\calH, \tilde{H}_1,\cdots, \tilde{H}_m$ are in
general position and $H_o:=\calH \cap \Delta \not=\varnothing$. In this case, $\Delta_+:=\Delta \cap \tilde{H}_{+}$ and $\Delta_-:=\Delta \cap
\tilde{H}_{+}$ are non-empty simple polytopes.

The simplicial complexes $K_{\Delta}, K_+, K_-$ associated to $\Delta, \Delta_+, \Delta_-$ to be defined as simplicial complexes defined on the vertex set $\widetilde{[m]}:=[m] \cup \{o\}$:
\begin{eqnarray*}
K_{\Delta}&:=& \{ \sigma \subset \widetilde{[m]} \ |\ \sigma \subset [m] \ \mbox{and}\ \cap_{i\in \sigma} H_i \not=\varnothing\} \cup \{\varnothing\}\\ 
K_+&:=&  \{ \sigma \subset \widetilde{[m]} \ |\  \cap_{i\in \sigma} (H_i\cap \Delta_+) \not=\varnothing\} \cup \{\varnothing\}\\
K_-&:=&   \{ \sigma \subset \widetilde{[m]} \ |\  \cap_{i\in \sigma} (H_i\cap \Delta_-) \not=\varnothing\} \cup \{\varnothing\}.
\end{eqnarray*}
%Let \begin{eqnarray*}
%\sfA_+ &:=& \{ i \in [m] \ |\ H_i \cap \Delta_+ = H_i\}\\
%\sfA_- &:=& \{ i \in [m] \ |\ H_i \cap \Delta_- = H_i\}\\
%\sfB &:=& \{ i \in [m] \ |\ H_i \cap H_o \not=\varnothing\}
%\end{eqnarray*}
%so that $\widetilde{[m]} = \sfA_+ \sqcup \sfA_- \sqcup \sfB \sqcup \{o\}$.
\end{defn}
%%%%%%%%%%%%%%%%%%%%%%%%%%%%%%%%%%%%%%%%%%
\begin{lem}\label{gh}
\begin{eqnarray}
K_+ \cap K_- &=& \sstar_{K_+\cup K_-}(o)= \sstar_{K_+}(o) = \sstar_{K_-}(o) \label{g}\\
(K_+ \cup K_-) \backslash K_{\Delta} &=&  O_{K_+\cup K_-}(o) = O_{K_+}(o)=O_{K_-}(o) \label{h}
\end{eqnarray}
\end{lem}
\begin{proof}
By definition, $\sigma \in K_+ \cap K_-$ iff $\sigma = \varnothing$ or $(\cap_{i\in \sigma} H_i) \cap \Delta_+ \cap \Delta_- \not= \varnothing$. Since $\Delta_+
\cap \Delta_- = H_o$, $\sigma \in K_+ \cap K_-$ iff $\sigma = \varnothing$ or $(\cap_{i\in \sigma} H_i \cap \Delta_+) \cap H_o = (\cap_{i\in \sigma} H_i \cap
\Delta_-) \cap H_o\not= \varnothing$. Therefore
\[
K_+ \cap K_- = \underbrace{\{ \sigma \in K_+ \ |\ \sigma \cup \{o\} \in K_+ \}}_{\sstar_{K_+}(o)} = \underbrace{\{ \sigma \in K_- \ |\ \sigma \cup \{o\} \in K_- \}}_{\sstar_{K_-}(o)}.
\]
By definition and $\Delta_+ \cup \Delta_- = \Delta$, $\sigma \in (K_+ \cup K_-) \backslash K_{\Delta}$ iff $\sigma \in K_+ \cup K_-$ and $o \in \sigma$. Thus
\[
(K_+ \cup K_-) \backslash K_{\Delta} = \{ \sigma \subset \tilde{[m]}\ |\ o \in \sigma, \mbox{ and }\sigma \in K_+ \cup K_-  \} = O_{K_+\cup K_-}(o).
\]
On the other hand, when $o \in \sigma$, $\sigma \in K_+$ iff $\sigma \in K_-$. Indeed, $\cap_{i\in \sigma} (H_i\cap \Delta_+) = (\cap_{i\in \sigma} H_i ) \cap
H_o = \cap_{i\in \sigma} (H_i\cap \Delta_-)$ if $o\in \sigma$. Thus $O_{K_+\cup K_-}(o) = O_{K_+}(o)=O_{K_-}(o)$.
\end{proof}
%%%%%%%%%%%%%%%%%%%%%%%%%%%%%%%%%%%%%%%%%%
\begin{thm}\label{P1}
If $(\Delta, \tilde{H}_o)$ is a generic cut, then $K_{\Delta}$ is the strong connected sum $K_+ ~\#^Z K_-$ where $Z = O_{K_+\cup K_-}(o)$.
\end{thm}
%%%%%%%%%%%%%%%%%%%%%%%%%%%%%%%%%%%%%%%%%%
\begin{proof}
From Lemma \ref{gh}, it is clear that $K_{\Delta}$ is the connected sum $K_+ ~\#^Z K_-$. We need to show $O_{K_{\pm}}(o) = W\backslash
(\overline{K_{\pm}\backslash W})$ where $W:=K_+ \cap K_- = \sstar_{K_+}(o) = \sstar_{K_-}(o)$ (See Lemma \ref{gh}). Suppose $\tau \in O_{K_+}(o)$. Since $\{o\}
\cup \sigma \not\in K_+$ for all $\sigma \in K_+\backslash W$, we have $\tau \cup \sigma \not\in K_+$ for all $\sigma \in K_+\backslash W$. Thus $O_{K_+}(o)
\subset W\backslash (\overline{K_+\backslash W})$ (See Lemma \ref{basiclemma}). To prove $W\backslash (\overline{K_+\backslash W}) \subset O_{K_+}(o)$, we show
that $\tau \in \sstar_{K_+}(o) \backslash O_{K_+}(o)$ implies $\tau \in \overline{K_+\backslash \sstar_{K_+}(o)}$. Since $\tau \in \sstar_{K_+}(o)$ and
$o\not\in\tau$, we have $\tau \subset \sfB$ such that $(\cap_{i\in \tau} H_i) \cap H_o \not=\varnothing$. Since the cutting is generic, $\dim \cap_{i\in \tau}
H_i \geq 1$ and $\cap_{i\in \tau} H_i$ has a vertex contained in $\Delta_+$ but not contained in $H_0$. Let $\cap_{i \in \sigma} H_i$ be such a vertex. Then
$\sigma \in K_+ \backslash W$. Since $\tau \subset \sigma$, $\tau \in \overline{K_+\backslash W}$. The same argument may be used to prove $O_{K_-}(o) = W\backslash
(\overline{K_-\backslash W})$.
\end{proof}
%%%%%%%%%%%%%%%%%%%%%%%%%%%%%%%%%%%%%%%%%%
\begin{lem}\label{lem+}
For $\sigma \subset \widetilde{[m]}$, let $F_{\sigma} := \cap_{i\in \sigma} H_i$. Let $Z=\{ \sigma \subset \widetilde{[m]} \ |\ F_{\sigma} \not=\varnothing
\mbox{ and } F_{\sigma} \subset \Delta_+\!\backslash H_o \}$.
\begin{eqnarray}
K_+ \cap K_{\Delta} &=& \overline{Z} \\ (K_+ \cup K_{\Delta}) \backslash K_- &=& Z
\end{eqnarray}
\end{lem}
\begin{proof}
$K_+ \cap K_{\Delta}$ consists of $\varnothing$ and $\sigma \subset [m]$ such that $F_{\sigma}\cap \Delta_+ \not=\varnothing$. Since $Z \subset K_+ \cap
  K_{\Delta}$, we have $\overline{Z} \subset K_+ \cap K_{\Delta}$. Suppose that $\sigma \in K_+ \cap K_{\Delta}$ and $\sigma \not \in Z$. Since $F_{\sigma}
  \not\subset \Delta_+\backslash H_o$ and $F_{\sigma}\cap \Delta_+ \not=\varnothing$, we have $F_{\sigma} \cap H_o \not=\varnothing$. Thus $\dim F_{\sigma} \geq
  1$ and so there is a vertex $F_{\tau}$ of $F_{\sigma}$ contained in $\Delta_+\backslash H_o$, which means $\tau \in Z$. Since $\sigma \subset \tau$, we have
  $\sigma \in \overline{Z}$. Thus $K_+ \cap K_{\Delta} \subset \overline{Z}$.

Since $F_{\sigma} \subset \Delta_+\backslash H_o$ iff $F_{\sigma} \cap \Delta_- = \varnothing$, it follows that  $(K_+ \cup K_{\Delta}) \backslash K_- =  Z$.
\end{proof}
%%%%%%%%%%%%%%%%%%%%%%%%%%%%%%%%%%%%%%%%%%
\begin{lem} Let $Z=\{ \sigma \subset \widetilde{[m]} \ |\  F_{\sigma} \not=\varnothing \mbox{ and } F_{\sigma} \subset \Delta_+\!\backslash H_o \}$.
\begin{eqnarray}
K_+\backslash \overline{Z} &=& O_{K_+}(o) \label{a}\\
K_{\Delta} \backslash \overline{Z} &=& \{ \sigma \subset \widetilde{[m]} \ |\  F_{\sigma} \not=\varnothing \mbox{ and } F_{\sigma} \subset \Delta_-\!\backslash H_o \}.
\end{eqnarray}
\end{lem}
\begin{proof}
By definition and (\ref{a}), $\sigma \in K_+\backslash \overline{Z}$ if and only if $o \in \sigma$ and $F_{\sigma}\not=\varnothing$. Thus $K_+\backslash
\overline{Z} = O_{K_+}(o)$. Also by definition and (\ref{a}), $\sigma \in K_{\Delta}\backslash \overline{Z}$ if and only if $F_{\sigma}\not=\varnothing$ and
$F_{\sigma} \subset \Delta_-\backslash H_o$.
\end{proof}
%%%%%%%%%%%%%%%%%%%%%%%%%%%%%%%%%%%%%%%%%%
\begin{thm}\label{P2}
Let $(\Delta, \tilde{H}_o)$ be a generic cut and let $Z=\{ \sigma \subset \widetilde{[m]} \ |\ F_{\sigma} \not=\varnothing \mbox{ and } F_{\sigma} \subset
\Delta_+\!\backslash H_o \}$.  Then $K_-$ is the strong connected sum $K_+~\#^Z K_{\Delta}$.
\end{thm}
%%%%%%%%%%%%%%%%%%%%%%%%%%%%%%%%%%
\begin{proof}
From Lemma \ref{lem+}, $K_-$ is the connected sum $K_+~\#^Z K_{\Delta}$. We only need to prove it is strong. Let $W:=\overline{Z} = K_+ \cap K_{\Delta}
$. First we show that $Z= W\backslash (\overline{K_+\backslash W}) = W \backslash\! \sstar_{K_+}(o)$. Suppose $\sigma \in Z$. If $\sigma \in \sstar_{K_+}(o)$,
then there must be $\tau \in O_{K_+}(o)$ such that $\sigma \subset\tau$. Since $o \in \tau$, we have $F_{\sigma} \cap H_o\not=\varnothing$ which contradicts
with $F_{\sigma} \subset \Delta_+\backslash H_o$. Thus $Z \subset W \backslash\! \sstar_{K_+}(o)$. On the other hand, if $\sigma \in W \backslash\!
\sstar_{K_+}(o)$, then $F_{\sigma}\cap \Delta_+\not=\varnothing$ and there is no vertex of $F_{\sigma}$ that lies on $H_o$. Therefore $F_{\sigma} \subset
\Delta_+\backslash H_o$, i.e. $\sigma \in Z$. Finally we show that $W\backslash (\overline{K_+\backslash W}) = W\backslash (\overline{K_{\Delta}\backslash
  W})$. Let $\varnothing\not=\sigma \in W \cap \overline{K_+\backslash W}$. Then $\sigma \subset [m]$ and $F_{\sigma} \cap H_o \not=\varnothing$. Thus $\dim
F_{\sigma}\geq 1$ and there is a vertex $F_{\tau}$ of $F_{\sigma}$ that lies in $\Delta_-\backslash H_o$. Since $\tau \in K_{\Delta} \backslash \overline{Z}$,
we have $\sigma \in \overline{K_+\backslash W}$. On the other hand, suppose that $\varnothing\not=\sigma \in W \cap \overline{K_{\Delta}\backslash W}$, then
$F_{\sigma} \cap \Delta_+\not=\varnothing$ and there is a vertex of $F_{\sigma}$ that lies in $\Delta_-\backslash H_o$. Thus $F_{\sigma} \cap
H_o\not=\varnothing$ which implies $\sigma \in \sstar_{K_+}(o)$.
\end{proof}

%%%%%%%%%%%%%%%%%%%%%%%%%%%%%%%%%%%%%%%%%%
%%%%%%%%%%%%%%%%%%%%%%%%%%%%%%%%%%%%%%%%%%
%%%%%%%%%%%%%%%%%%%%%%%%%%%%%%%%%%%%%%%%%%
\section{{\bf Stanley-Reisner Rings and Connected Sum}}
We study the algebraic structure of the Stanley-Reisner ring of the connected sum $K_1~\#^Z K_2$ defined in the previous section. The algebraic model is
the \emph{connected sum of rings} introduced and studied by Ananthnarayan-Avramov-Moore \cite{AAM}. In Section \ref{4-1}, we review the definitions and show
that the Stanley-Reisner ring $\ZZ[K_1~\#^Z K_2]$ is the connected sum of the Stanley-Reisner ring of $K_1$ and $K_2$. In Section \ref{canonical}, we study
the Gorensteinness of $\ZZ[K_1~\#^Z K_2]$ in terms of the ones of of $K_1$, $K_2$ and $K_1\cap K_2$ for strong connected sums. Here Corollary \ref{just} is
our motivation to define \emph{strong} connected sums. In Section \ref{4-3}, we discuss how those properties descend to Torsion algebras of Stanley-Reisner
rings.
%%%%%%%%%%%%%%%%%%%%%%%%%%%%%%%%%%%%%%%%%%
\subsection{Connected Sum of Rings}\label{4-1}
%%%%%%%%%%%%%%%%%%%%%%%%%%%%%%%%%%%%%%%%%%
\begin{defn}[Fiber Product and Connected Sum of Rings]\label{defconn}
Let $\epsilon_{\sfA}: \sfA \to \sfC$ and $\epsilon_{\sfB}:\sfB \to \sfC$ be ring homomorphisms. Then the \emph{fiber product} $\sfA\times_{\sfC}\sfB$ is the
subring of $\sfA \oplus \sfB$ defined by $\sfA\times_{\sfC}\sfB := \{(x,y) \in \sfA \oplus \sfB \ |\ \epsilon_{\sfA} (x) = \epsilon_{\sfB}(y)\}$.
%The canonical projections $\pi_{\sfA}$ and $\pi_{\sfB}$ of $\sfA\times_{\sfC}\sfB$ to $\sfA$ and $\sfB$ satisfies the usual universal property:
%\[
%\xymatrix{
%&& \sfA \ar[rd]^{\epsilon_{\sfA}} &  \\
%\forall\sfD\ar@/^1pc/[rru]^{\forall\alpha}\ar@/_1pc/[rrd]_{\forall\beta} \ar@{.>}[r]_{\exists!}&\sfA\times_{\sfC}\sfB\ar[ru]^{\pi_{\sfA}}\ar[rd]_{\pi_{\sfB}}&& \sfC\\
%&&\sfB \ar[ru]_{\epsilon_{\sfB}}  &
%}
%\]
Now take a $\sfC$-module $\sfV$ and regard it as a $\sfA$-module and a $\sfB$-module through $\epsilon_{\sfA}$ and $\epsilon_{\sfB}$. Consider the commutative diagram
\begin{equation} \label{connsumdiagram}
\xymatrix{
\sfV\ar[d]_{\iota_{\sfB}}\ar[r]^{\iota_{\sfA}}& \sfA \ar[d]^{\epsilon_{\sfA}}  \\
\sfB \ar[r]_{\epsilon_{\sfB}}  & \sfC
}
\end{equation}
where $\iota_{\sfA}$ and $\iota_{\sfB}$ are homomorphisms of $\sfA$-modules and $\sfB$-modules. The {\bf\emph{connected sum}} of the diagram \eqref{connsumdiagram} is given by
\[
\sfA ~\#_{\sfC}^{\sfV} \sfB := \frac{\sfA \times_{\sfC}\sfB }{ \{(\iota_{\sfA}(v), \iota_{\sfB}(v)) \in \sfA \oplus \sfB \ |\ v \in \sfV\}}.
\]
\end{defn}
%%%%%%%%%%%%%%%%%%%%%%%%%%%%%%%%%%%%%%%%%%
\begin{rem}
One may also view the definition of the connected sum of rings as arising via the following exact sequences:
\begin{eqnarray}
&&0 \longrightarrow \sfA\times_{\sfC}\sfB \longrightarrow \sfA\oplus \sfB \stackrel{(\epsilon_{\sfA},- \epsilon_{\sfB})}{\longrightarrow} \sfC\\
&& \ \ \ \ \ \ \ \ \ \ \ \  \sfV \longrightarrow \sfA\times_{\sfC}\sfB \longrightarrow \sfA ~\#_{\sfC}^{\sfV} \sfB \longrightarrow 0
\end{eqnarray}
\end{rem}
%%%%%%%%%%%%%%%%%%%%%%%%%%%%%%%%%%%%%%%%%%
\begin{thm}\label{thm1}
Let $\tilde{K}:=K_1\cup K_2$ and $W:=K_1\cap K_2$ where $K_1$ and $K_2$ are simplicial complexes on $[m]$. There is a natural isomorphism $\theta:
\ZZ[\tilde{K}] \to \ZZ[K_1]\times_{\ZZ[W]} \ZZ[K_2]$ defined by $\theta(r)=(\sff_1(r), \sff_2(r))$ where $\sff_1: \ZZ[\tilde{K}] \to \ZZ[K_1]$ and $\sff_2:
\ZZ[\tilde{K}] \to \ZZ[K_2]$ are the obvious quotient maps.
\end{thm}
\begin{proof}
Observe $\calZ_{\tilde{K}} = \calZ_{K_1} \cup \calZ_{K_2}$ and $\calZ_{W} = \calZ_{K_1} \cap \calZ_{K_2}$. Then we can apply the Mayer-Vietoris Sequence for
$\sfT$-equivariant cohomology. Since there are no odd degree classes, the sequence splits into short exact sequences. By Theorem \ref{DJ}, we have
\[
0 \to \ZZ[\tilde{K}] \stackrel{(\sff_1,\sff_2)}{\longrightarrow}  \ZZ[K_1]\oplus \ZZ[K_2] \stackrel{(\sfg_1,-\sfg_2)}{\longrightarrow} \ZZ[W] \to 0
\]
where $\sfg_1$ and $\sfg_2$ are the obvious quotient maps. The kernel $(\sfg_1,\sfg_2)$ is the fiber product and so $\theta$ gives the isomorphism.
\end{proof}
%%%%%%%%%%%%%%%%%%%%%%%%%%%%%%%%%%%%%%%%%%
\begin{thm}\label{thm2}
Let $K_1~\#^Z K_2$ be a connected sum. Then there is a natural isomorphism $\xi: \ZZ[K_1]~\#_{\ZZ[W]}^{\calJ_Z} \ZZ[K_2] \to \ZZ[K_1~\#^Z K_2]$
where $\calJ_Z$ is the ideal in $\ZZ[W]$ generated by $x_{\sigma}, \sigma \in Z$.
\end{thm}
\begin{proof}
Let $K:= K_1~\#^Z K_2 = \Delt_Z(\tilde{K})$. The relative cohomology sequence for the pair $(\calZ_{\tilde{K}}, \calZ_{K})$ splits into short exact
sequence. By Theorem \ref{DJ} and Theorem \ref{thm1}, we obtain
\[
0 \to \calI_Z \stackrel{\theta|_{\calI_Z}}{\longrightarrow} \ZZ[K_1]\times_{\ZZ[W]} \ZZ[K_2] \stackrel{\sfh\circ\theta^{-1}}{\longrightarrow} \ZZ[K] \to 0
\]
where $\sfh: \ZZ[\tilde{K}] \to \ZZ[K]$ is the obvious quotient map and $\calI_Z$ is the ideal in $\ZZ[\tilde{K}]$ generated by $x_{\sigma},\sigma\in Z$. Since
$O_{\tilde{K}}(Z) \subset W$, $\sfj: \calI_Z \to \calJ_Z, x_{\sigma} \mapsto x_{\sigma}$ is an isomorphism of $\ZZ[x_1,\cdots, x_m]$-modules.  Since the
connected sum $\ZZ[K_1]~\#_{\ZZ[W]}^{\calJ_Z} \ZZ[K_2]$ is defined to be $\ZZ[K_1]\times_{\ZZ[W]} \ZZ[K_2]/ \theta\circ \sfj^{-1}(\calJ_Z)$, the map $\xi$ is
the isomorphism induced from $\sfh\circ\theta^{-1}$.
\end{proof}
%%%%%%%%%%%%%%%%%%%%%%%%%%%%%%%%%%%%%%%%%%
\subsection{Connected sum of Gorenstein rings}\label{canonical}
%%%%%%%%%%%%%%%%%%%%%%%%%%%%%%%%%%%%%%%%%%
%%%%%%%%%%%%%%%%%%%%%%%%%%%%%%%%%%%%%%%%%%
Let $W$ be a subcomplex of a simplicial complex $K$ on $[m]$. Let $\calI_{K\backslash W}$ be a kernel of the quotient map $\ZZ[K] \to \ZZ[W]$. 
\begin{lem}\label{la}
The annihilator $(0:_{\ZZ[K]} \calI_{K\backslash W})$ is generated by $x_{\sigma}, \sigma \in W \backslash (\overline{K\backslash W})$.
\end{lem}
%%%%%%%%%%%%%%%%%%%%%%%%%%%%%%%%%%%%%%%%%%
\begin{proof}
The annihilator is generated by $x_{\sigma}$ where $\sigma \in K$ s.t. $\sigma \cup \tau \not\in K, \forall \tau \in K \backslash W$. The claim is a corollary of Lemma \ref{basiclemma}.
\end{proof}
%%%%%%%%%%%%%%%%%%%%%%%%%%%%%%%%%%%%%%%%%%
The following is a basic fact about the canonical module of a Cohen-Macaulay ring \cite[Theorem 3.3.7]{BHComm}:
\begin{lem}\label{GorCano}
Suppose that $W$ and $K$ are pure with the same dimension. If $K$ is Gorenstein and $W$ is Cohen-Macaulay, then $(0:_{\ZZ[K]} \calI)$ is a canonical module of $\ZZ[W]$.
\end{lem}
%%%%%%%%%%%%%%%%%%%%%%%%%%%%%%%%%%%%%%%%%%
From \cite{AAM}, we have the following theorem.
\begin{thm}
In the definition \ref{defconn}, $\sfA ~\#_{\sfC}^{\sfV} \sfB$ is Gorenstein if $\sfA$ and $\sfB$ are Gorenstein, $\sfC$ is Cohen-Macaulay and $\sfV$ is a canonical module of $\sfC$.
\end{thm}
%\vspace{7pt}
%\fbox{\parbox[r]{5.9in}{NEED PROOF HERE
%}}
%\vspace{7pt}
As a corollary, together with Lemma \ref{la} and \ref{GorCano}, we have 
%%%%%%%%%%%%%%%%%%%%%%%%%%%%%%%%%%%%%%%%%%
\begin{cor}\label{just}
Let $K_1$ and $K_2$ are simplicial complexes on $[m]$ such that $K_1$, $K_2$ and $W:=K_1 \cup K_2$ are pure with the same dimension. Assume that $K_1, K_2$ are
Gorenstein and $W$ is Cohen-Macaulay. If $K_1 ~\#^ZK_2$ is a strong connected sum, then $\ZZ[K_1 ~\#^Z K_2]$ is Gorenstein.
\end{cor}
The above corollary is the algebraic motivation to have Definition \ref{defconn} of the strong connected sum.
%%%%%%%%%%%%%%%%%%%%%%%%%%%%%%%%%%%%%%%%%%
%%%%%%%%%%%%%%%%%%%%%%%%%%%%%%%%%%%%%%%%%%
%%%%%%%%%%%%%%%%%%%%%%%%%%%%%%%%%%%%%%%%%%
\subsection{Tor algebra of connected sums}\label{4-3}
%%%%%%%%%%%%%%%%%%%%%%%%%%%%%%%%%%%%%%%%%%
Let $K_1~\#_W^Z K_2$ be a connected sum and let $\tilde{K}=K_1\cup K_2$ and $K=K_1~\#_W^ZK_2$. Let $[m]=\{1,\cdots, m\}$ be the vertex set of
$\tilde{K}$. Theorem \ref{thm1} and Theorem \ref{thm2} imply that there are two short exact sequences of algebras and modules over $\ZZ[x_1,\cdots, x_n]$:
\begin{equation}
0 \to \ZZ[\tilde{K}] \to \ZZ[K_1]\oplus \ZZ[K_2] \to \ZZ[W] \to 0
\end{equation}
\begin{equation}
0 \to \calI_Z \to \ZZ[\tilde{K}] \to \ZZ[K] \to 0
\end{equation}
%%%%%%%%%%%%%%%%%%%%%%%%%%%%%%%%%%%%%%%%%%
Consider an integer $n\times m$ matrix $B$ of rank $n$. The choice of such $B$ bijectively corresponds to a choice of a surjective map $\sfT:=U(1)^m \to
\sfR:=U(1)^n$. Denote $\ZZ[\sfT^*]:=\ZZ[x_1,\cdots, x_m]$. Let $u_i:=\sum_{j=1}^m B_{ij} x_j$ and denote $\ZZ[\sfR^*]:=\ZZ[u_1,\cdots, u_n] \subset
\ZZ[\sfT^*]$. Consider the Koszul complex $\calK^{\sfR}$ given by the exterior algebra generated by $\xi_1,\cdots, \xi_n$ over $\ZZ[\sfR^*]$. By tensoring
$\calK^{\sfR}$ to the short exact sequences above, we obtain the short exact sequences of complexes, therefore we have the long exact sequences:
\begin{equation}
\cdots \to \Tor_{i+1}^{\ZZ[\sfR^*]}(\ZZ[W], \ZZ) \to \Tor_{i}^{\ZZ[\sfR^*]}(\ZZ[\tilde{K}], \ZZ) \to \Tor_{i}^{\ZZ[\sfR^*]}(\ZZ[K_1 ], \ZZ)\oplus
\Tor_{i}^{\ZZ[\sfR^*]}(\ZZ[K_2 ], \ZZ)\to \Tor_{i}^{\ZZ[\sfR^*]}(\ZZ[W], \ZZ) \to \cdots
\end{equation}
\begin{equation}
\cdots \to \Tor_{i+1}^{\ZZ[\sfR^*]}(\ZZ[K], \ZZ) \to \Tor_{i}^{\ZZ[\sfR^*]}(\calI_Z, \ZZ) \to \Tor_{i}^{\ZZ[\sfR^*]}(\ZZ[\tilde{K}], \ZZ) \to \Tor_{i}^{\ZZ[\sfR^*]}(\ZZ[K], \ZZ) \to \cdots 
\end{equation} 
The following claims can be easily observed:
%%%%%%%%%%%%%%%%%%%%%%%%%%%%%%%%%%%%%%%%%%
\begin{lem}
Suppose that $\Tor_{1}^{\ZZ[\sfR^*]}(\ZZ[W], \ZZ)=0$. Then $\Tor_{1}^{\ZZ[\sfR^*]}(\ZZ[\tilde{K}], \ZZ)=0$ if and only if $\Tor_{1}^{\ZZ[\sfR^*]}(\ZZ[K_1 ],
\ZZ)=\Tor_{1}^{\ZZ[\sfR^*]}(\ZZ[K_2 ], \ZZ)=0$. In this case,
\[
\Tor_{0}^{\ZZ[\sfR^*]}(\ZZ[\tilde{K}], \ZZ) = \Tor_{0}^{\ZZ[\sfR^*]}(\ZZ[K_1 ], \ZZ)\times_{\Tor_{0}^{\ZZ[\sfR^*]}(\ZZ[W], \ZZ)}\Tor_{0}^{\ZZ[\sfR^*]}(\ZZ[K_2 ], \ZZ).
\]
\end{lem}
%%%%%%%%%%%%%%%%%%%%%%%%%%%%%%%%%%%%%%%%%%
\begin{lem}\label{fiberprodtor}
If $\Tor_{1}^{\ZZ[\sfR^*]}(\ZZ[K_1], \ZZ)=\Tor_{1}^{\ZZ[\sfR^*]}(\ZZ[K_2], \ZZ)=\Tor_{1}^{\ZZ[\sfR^*]}(\ZZ[K], \ZZ)=\Tor_{1}^{\ZZ[\sfR^*]}(\ZZ[W], \ZZ)=0$, then 
\[
\Tor_{0}^{\ZZ[\sfR^*]}(\ZZ[K], \ZZ) = \Tor_{0}^{\ZZ[\sfR^*]}(\ZZ[K_1 ], \ZZ)~\#^{\Tor_{0}^{\ZZ[\sfR^*]}(\calI_Z, \ZZ)}_{\Tor_{0}^{\ZZ[\sfR^*]}(\ZZ[W], \ZZ)}\Tor_{0}^{\ZZ[\sfR^*]}(\ZZ[K_- ], \ZZ).
\]
\end{lem}
%%%%%%%%%%%%%%%%%%%%%%%%%%%%%%%%%%%%%%%%%%
\begin{rem}
By Proposition 2.3 \cite{FranzPuppe07}, $\Tor_1=0$ implies $\Tor_i=0$ for all $i>0$. Therefore, in the above lemmata, we actually have $\Tor_{0}^{\ZZ[\sfR^*]}(\ZZ[\tilde{K}], \ZZ)=\Tor_*^{\ZZ[\sfR^*]}(\ZZ[\tilde{K}], \ZZ)$ and $\Tor_*^{\ZZ[\sfR^*]}(\ZZ[K], \ZZ) =\Tor_{0}^{\ZZ[\sfR^*]}(\ZZ[K], \ZZ)$.
\end{rem}
%%%%%%%%%%%%%%%%%%%%%%%%%%%%%%%%%%%%%%%%%%
\begin{lem}\label{badgene}
Suppose that $K_1,K_2$ are defined by a generic cut of a polytope and $\Tor_{1}^{\ZZ[\sfR^*]}(\ZZ[W], \ZZ)=0$. If $\Tor_{1}^{\ZZ[\sfR^*]}(\ZZ[K], \ZZ)=0$, then
$\Tor_{1}^{\ZZ[\sfR^*]}(\ZZ[K_1 ], \ZZ)=\Tor_{1}^{\ZZ[\sfR^*]}(\ZZ[K_2 ], \ZZ)=0$.
\end{lem}
\begin{proof}
In this case, observe that $\calI_Z \cong \ZZ[W]$ as $\ZZ[\sfT^*]$-modules. Thus $\Tor_{1}^{\ZZ[\sfR^*]}(\ZZ[W], \ZZ)=\Tor_{1}^{\ZZ[\sfR^*]}(\ZZ[K], \ZZ)=0$
implies $\Tor_{1}^{\ZZ[\sfR^*]}(\ZZ[\tilde{K}], \ZZ)=0$ and hence $\Tor_{1}^{\ZZ[\sfR^*]}(\ZZ[K_1 ], \ZZ)=\Tor_{1}^{\ZZ[\sfR^*]}(\ZZ[K_2 ], \ZZ)=0$.
\end{proof}
%%%%%%%%%%%%%%%%%%%%%%%%%%%%%%%%%%%%%%%%%%
\begin{rem}
The opposite statement of Lemma \ref{badgene} is not true. We give an example which shows that $\Tor_{1}^{\ZZ[\sfR^*]}(\ZZ[W],
\ZZ)=\Tor_{1}^{\ZZ[\sfR^*]}(\ZZ[K_1], \ZZ)=\Tor_{1}^{\ZZ[\sfR^*]}(\ZZ[K_2], \ZZ)=0$ does not imply $\Tor_{1}^{\ZZ[\sfR^*]}(\ZZ[K], \ZZ)=0$.

Consider the following simplicial complexes
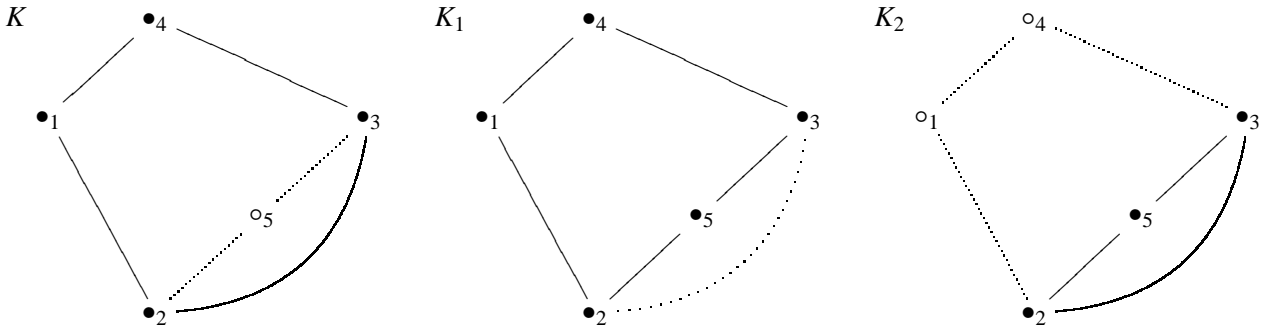
\[
K \xymatrix{
&\bullet_4\ar@{-}[rrd]&&\\
\bullet_1\ar@{-}[ru]\ar@{-}[rdd]&&&\bullet_3\ar@{-}@/^2pc/[lldd]\\
&&\circ_5\ar@{.}[ur] \ar@{.}[ld]&\\
&\bullet_2&&\\
}
\ \ \ \ \ \ 
K_1 \xymatrix{
&\bullet_4\ar@{-}[rrd]&&\\
\bullet_1\ar@{-}[ru]\ar@{-}[rdd]&&&\bullet_3\ar@{.}@/^2pc/[lldd]\\
&&\bullet_5\ar@{-}[ur] \ar@{-}[ld]&\\
&\bullet_2&&\\
}
\ \ \ \ \ \ 
K_2 \xymatrix{
&\circ_4\ar@{.}[rrd]&&\\
\circ_1\ar@{.}[ru]\ar@{.}[rdd]&&&\bullet_3\ar@{-}@/^2pc/[lldd]\\
&&\bullet_5\ar@{-}[ur] \ar@{-}[ld]&\\
&\bullet_2&&\\
}
\]
$K$ is a strong connected sum of $K_+$ and $K_-$ along $W:=K_1 \cap K_2$. Consider the following $2\times 5$ matrix $B$:
\[
B = 
\begin{pmatrix}
1&0&-2& 0&-1\\
0&2&0 &-1&1
\end{pmatrix}
\]
By direct computation (we used \emph{Macaulay2}), we find that  
\[
\Tor_{1}^{\ZZ[\sfR^*]}(\ZZ[W], \ZZ)=\Tor_{1}^{\ZZ[\sfR^*]}(\ZZ[K_1], \ZZ)=\Tor_{1}^{\ZZ[\sfR^*]}(\ZZ[K_2], \ZZ)=0
\]
but $\Tor_{1}^{\ZZ[\sfR^*]}(\ZZ[K], \ZZ)\not=0$. 

The above example comes from cutting a labeled polytope $(\Delta,\sfb)$ that corresponds to the direct product of weighted projective space, $\CP^1_{12} \times \CP^1_{12}$: 
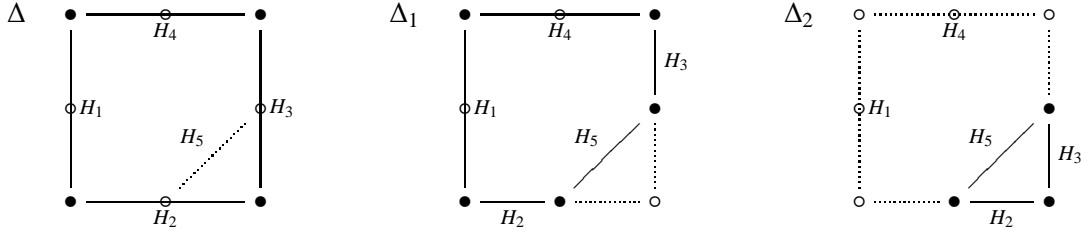
\[
\Delta\ \ \ \ \xymatrix{
\bullet\ar@{-}[rr]_{H_4}\ar@{-}[dd]^{H_1}&\circ&\bullet\ar@{-}[dd]^{H_3}\\
\circ &&\circ\\
\bullet\ar@{-}[rr]_{H_2} &\circ\ar@{.}[ru]^{H_5}& \bullet 
} 
\ \ \ \ \ \ \ \ \ \ \ \ 
\Delta_1\ \ \ \ \xymatrix{
\bullet\ar@{-}[rr]_{H_4}\ar@{-}[dd]^{H_1}&\circ&\bullet\ar@{-}[d]^{H_3}\\
\circ &&\bullet\ar@{.}[d]\\
\bullet\ar@{-}[r]_{H_2} &\bullet\ar@{-}[ru]^{H_5}\ar@{.}[r]& \circ 
} 
\ \ \ \ \ \ \ \ \ \ \ \ 
\Delta_2\ \ \ \ \xymatrix{
\circ\ar@{.}[rr]_{H_4}\ar@{.}[dd]^{H_1}&\circ&\circ\ar@{.}[d]\\
\circ &&\bullet\ar@{-}[d]^{H_3}\\
\circ\ar@{.}[r]&\bullet\ar@{-}[ru]^{H_5}\ar@{-}[r]_{H_2} & \bullet 
} 
\]
The polytope $\Delta$ is labeled by $\sfb=(1,2,2,1)$, the cutting facet $H_5$ is labeled by $1$, and the matrix $B$ actually corresponds to the extended
$B$-matrix $\tilde{B}$ in the notation of Section \ref{cuttingprocess}.
\end{rem}

%%%%%%%%%%%%%%%%%%%%%%%%%%%%%%%%%%%%%%%%%%
%   		     BIBLIOGRAPHY 
%%%%%%%%%%%%%%%%%%%%%%%%%%%%%%%%%%%%%%%%%%
\bibliography{references}{}
\bibliographystyle{plain}
%%%%%%%%%%%%%%%%%%%%%%%%%%%%%%%%%%%%%%%%%%
%%%%%%%%%%%%%%%%%%%%%%%%%%%%%%%%%%%%%%%%%%
\end{document}